\documentclass[12pt,a4paper]{article}
\usepackage{amssymb}
\usepackage{amsmath}
\usepackage{amscd}
\usepackage{latexsym}
\usepackage{times}
\usepackage[usenames]{color}
\usepackage{ifthen}
\usepackage{theorem}
\usepackage{enumerate,multirow}
\definecolor{sepia}{rgb}   {0.8,   0.5,   0.} 
\usepackage[colorlinks,citecolor=sepia,linkcolor=sepia,
            bookmarksopen,
            bookmarksnumbered
           ]{hyperref}

\newtheorem{theorem}{Theorem}[section]
\newtheorem{proposition}[theorem]{Proposition}
\newtheorem{lemma}[theorem]{Lemma}
\newtheorem{corollary}[theorem]{Corollary}
\newtheorem{Remark}[theorem]{Remark}
\newtheorem{definition}[theorem]{Definition}
\newtheorem{Assumption}{Assumption}
\numberwithin{equation}{section}

\newenvironment{proof}{\emph{Proof.} }{\mbox{$$}\hfill$\Box$}
\newenvironment{remark}
          {\begin{Remark} \parindent=0pt\rm }
          {\hskip 12bp$\square$\par\addvspace{8bp}\end{Remark}}
\newenvironment{assumption}
          {\begin{Assumption} \parindent=0pt\rm 
           \abovedisplayskip = 0.6 \abovedisplayskip
           \belowdisplayskip=\abovedisplayskip}
          {\end{Assumption}}
\newcommand{\ass}[2]{%
\par\medskip\par
\noindent\fbox{\begin{minipage}{0.985\textwidth}
      \begin{assumption}\label{#1}{#2}
      \end{assumption}%
    \end{minipage}}\par\bigskip\par%
}

\newcommand  {\N}{{\mathbb N}}

\newcommand  {\R}{{\mathbb R}}

\newcommand{\Div}{\operatorname{\mathsf{div}}}
\newcommand{\grad}{\operatorname{\boldsymbol{\mathsf{grad}}}}
\newcommand{\curl}{\operatorname{\boldsymbol{\mathsf{curl}}}}
\newcommand{\curltd}{\operatorname{\mathsf{curl}}}
\renewcommand{\div}{\operatorname{\mathsf{div}}}
\newcommand{\supp}{\operatorname{{supp}}}

\newcommand{\extd}[1][]{\operatorname{\mathsf{d}}_{#1}}

\newcommand{\delop}[1][]{\operatorname{\delta}_{#1}}

\newcommand{\hodge}[1][]{\operatorname{\star}_{#1}}
\newcommand{\contr}{\mathbin{\lrcorner}}

\newcommand*{\trace}[1]{\operatorname{tr}_{#1}}

\newcommand{\Poinc}{\mathsf{R}}

\newcommand{\bff}{{\mathbf{f}}}

\newcommand{\bfn}{{\mathbf{n}}}

\newcommand{\bfq}{{\mathbf{q}}}

\newcommand{\bfu}{{\mathbf{u}}}
\newcommand{\bfv}{{\mathbf{v}}}
\newcommand{\bfw}{{\mathbf{w}}}
\newcommand{\bfx}{{\mathbf{x}}}

\newcommand{\bfz}{{\mathbf{z}}}

\providecommand{\BH}{{\boldsymbol{H}}}

\newcommand{\phibf}{\boldsymbol{\phi}}

\newcommand{\psibf}{\boldsymbol{\psi}}

\newcommand{\Phibf}{\boldsymbol{\Phi}}

\providecommand{\FM}{\mathfrak{M}}

\providecommand{\bbN}{\mathbb{N}}

\providecommand{\bbR}{\mathbb{R}}

\providecommand*{\Norm}[1]{\left\|{#1}\right\|} 
\newcommand*{\SP}[2]{\left({#1},{#2}\right)} 

\newcommand{\contemb}{\hookrightarrow}              

\newcommand*{\Op}[1]{\mathsf{#1}} 

\newcommand{\iti}[1]{\noindent\mbox{{\em (\romannumeral #1)\/}}}
\newcommand{\db}[1]{_{\raise-0.3ex\hbox{$\scriptstyle #1$}}}

\newcommand{\on}[1]{\!\left.\vphantom{|_|}\right|_{#1}}  

\newcommand{\overcirc}[2]{%
  \hspace{-0.7ex}\hspace{#1ex}{\raisebox{#2ex}{$\scriptstyle\circ$}}\hspace{-#1ex}}
\newcommand{\ocH}{\mathchoice
	{\overcirc{1.6}{1.7}H}{\overcirc{1.6}{1.6}H}
	{\overcirc{1.3}{1.1}H}{\overcirc{1.3}{1.1}H}}
\newcommand{\ocBH}{\overcirc{1.7}{1.6}\BH}
\newcommand{\ocY}{\overcirc{1.3}{1.6}Y}
\newcommand{\ocCV}{\mathchoice
	{\overcirc{1.25}{1.7}\mathcal{V}}{\overcirc{1.25}{1.6}\mathcal{V}}
	{\overcirc{1.0}{1.1}\mathcal{V}}{\overcirc{1.25}{1.1}\mathcal{V}}}
\newcommand{\ocCY}{\overcirc{1.35}{1.6}\mathcal{Y}}
\newcommand{\ocCZ}{\mathchoice
	{\overcirc{1.5}{1.7}\mathcal{Z}}{\overcirc{1.5}{1.6}\mathcal{Z}}
	{\overcirc{1.2}{1.1}\mathcal{Z}}{\overcirc{1.2}{1.1}\mathcal{Z}}}

\newcommand{\Dom}{\Omega}

\newcommand*{\Ltwo}[1][\Dom]{L^{2}({#1})}
\newcommand*{\NLtwo}[2][\Dom]{\Norm{#2}_{\Ltwo[#1]}}
\newcommand*{\SPLtwo}[3][\Dom]{\SP{#2}{#3}_{\Ltwo[#1]}}

\newcommand*{\zbHone}[1][\Dom]{\zbHm[#1]{1}}

\newcommand*{\Hm}[2][\Dom]{H^{#2}({#1})}
\newcommand*{\HHm}[2][\Dom]{\boldsymbol{H}^{#2}({#1})}
\newcommand*{\zbHm}[2][\Dom]{\ocH^{#2}({#1})}

\newcommand*{\Hdiv}[1][\Dom]{\boldsymbol{H}(\Div,{#1})}
\newcommand*{\zbHdiv}[1][\Dom]{\ocBH%
  (\Div,{#1})}
\newcommand*{\kHdiv}[1][\Dom]{\boldsymbol{H}(\Div0,{#1})}

\newcommand*{\Hcurl}[1][\Dom]{\boldsymbol{H}(\curl,{#1})}
\newcommand*{\Hcurlm}[2][\Dom]{\boldsymbol{H}^{#2}(\curl,{#1})}
\newcommand*{\zbHcurl}[1][\Dom]{\ocBH%
  (\curl,{#1})}
\newcommand*{\kHcurl}[1][\Dom]{\boldsymbol{H}(\curl0,{#1})}

\newcommand*{\zbHcurltd}[1][\Dom]{\ocBH%
  (\curltd,{#1})}

\newcommand{\SPN}[3][\Dom]{\SP{#2}{#3}_{0,{#1}}}

\newcommand*{\FCinf}[2][\Dom]{C^{\infty}({#1},\Lambda^{#2})}

\newcommand*{\FLtwo}[2][\Dom]{L^{2}({#1},\Lambda^{#2})}
\newcommand*{\NFLtwo}[3][\Dom]{\Norm{#3}_{\FLtwo[#1]{#2}}}

\newcommand*{\FHm}[3][\Dom]{H^{#3}({#1},\Lambda^{#2})}
\newcommand*{\zbFHm}[3][\Dom]{\ocH^{#3}({#1},%
  \Lambda^{#2})}

\newcommand*{\FHd}[2][\Dom]{H(\extd[#2],{#1})}
\newcommand*{\zbFHd}[2][\Dom]{\ocH(\extd[#2],{#1})}

\newcommand*{\zbkFHd}[2][\Dom]{\ocH(\extd[#2]0,{#1})}
\newcommand*{\NFHd}[3][\Dom]{\Norm{#3}_{\FHd[#1]{#2}}}
\newcommand*{\FHdr}[3][\Dom]{H^{#3}(\extd[#2],{#1})}

\newcommand*{\zbFXd}[2][\Dom]{\ocY(\extd[#2],{#1})}

\newcommand{\HX}[2][\mesh]{X({#1},\Lambda^{#2})}
\newcommand{\HXt}[3][\mesh]{\widetilde{X}_{#3}({#1},\Lambda^{#2})}

\newcommand{\NHX}[3][\mesh]{\Norm{#3}_{\HX[#1]{#2}}}

\newcommand{\HS}[2][\mesh]{S({#1},\Lambda^{#2})}

\newcommand{\NHS}[3][\mesh]{\Norm{#3}_{\HS[#1]{#2}}}

\newcommand{\mesh}{\FM}
\newcommand{\defaultmesh}{\mesh}

\newcommand*{\Faces}[2][\mesh]{\mathfrak{F}_{#2}({#1})}

\newcommand*{\Pol}{\mathcal{P}}

\newcommand*{\FPol}[2]{\Pol_{#1}(\Lambda^{#2})}

\newcommand*{\Wpl}[3][\defaultmesh]{%
  \ifthenelse{\equal{#1}{}}%
  {\mathcal{V}^{#2}_{#3}}%
  {\mathcal{V}^{#2}_{#3}({#1})}}
\newcommand*{\zbWpl}[3][\defaultmesh]{%
  \ifthenelse{\equal{#1}{}}%
  {\ocCV^{#2}_{#3}}%
  {\ocCV^{#2}_{#3}({#1})}}
\newcommand*{\kWpl}[3][\defaultmesh]{%
  \ifthenelse{\equal{#1}{}}%
  {\mathcal{V}^{#2}_{#3}(\extd0)}%
  {\mathcal{V}^{#2}_{#3}(\extd0,{#1})}}
\newcommand*{\zbkWpl}[3][\defaultmesh]{%
  \ifthenelse{\equal{#1}{}}%
  {\ocCV^{#2}_{#3}(\extd0)}%
  {\ocCV^{#2}_{#3}(\extd0,{#1})}}

\newcommand{\Xpl}[3][\defaultmesh]{%
  \ifthenelse{\equal{#1}{}}%
  {\mathcal{Y}^{#2}_{#3}}%
  {\mathcal{Y}^{#2}_{#3}({#1})}}
\newcommand{\zbXpl}[3][\defaultmesh]{%
  \ifthenelse{\equal{#1}{}}%
  {\ocCY^{#2}_{#3}}%
  {\ocCY^{#2}_{#3}({#1})}}
\newcommand{\zbZpl}[3][\defaultmesh]{%
  \ifthenelse{\equal{#1}{}}%
  {\ocCZ^{#2}_{#3}}%
  {\ocCZ^{#2}_{#3}({#1})}}

\newcommand*{\PO}[2]{\pi^{#1}_{#2}}
\newcommand*{\LPO}[3][\defaultcell]{\pi^{#2}_{#3,{#1}}}

\newcommand{\E}{\mathbf{E}}
\renewcommand{\H}{\mathbf{H}}

\begin{document}
\hfuzz 2pt

\title{Discrete compactness for the {$p$}-version of discrete differential forms}

\author{{Daniele Boffi}%
  \thanks{Dipartimento di Matematica ``F.~Casorati'', Universit\`a di Pavia,
    I-27100 Pavia, Italy, daniele.boffi@unipv.it}
  \and
  {Martin Costabel}
  \thanks{IRMAR, Universit\'e de Rennes 1, 35042
    Rennes, France, martin.costabel@univ-rennes1.fr}
  \and
  {Monique Dauge}
  \thanks{IRMAR, Universit\'e de Rennes 1, 35042
    Rennes, France, monique.dauge@univ-rennes1.fr}
  \and
  {Leszek Demkowicz}
  \thanks{The Institute for Computational Engineering and Sciences,
    The University of Texas at Austin, Austin, TX 78712, USA, leszek@ices.utexas.edu}
  \and
  {Ralf Hiptmair}
  \thanks{{SAM, ETH Z\"urich, CH-8092 Z\"urich,
      hiptmair@sam.math.ethz.ch}}
}

\maketitle

\begin{abstract}
  In this paper we prove the discrete compactness property for a wide class of $p$
  finite element approximations of non-elliptic variational eigenvalue problems in two
  and three space dimensions. In a very general framework, we find sufficient
  conditions for the $p$-version of a generalized discrete compactness property,
  which is formulated in the setting of discrete differential forms of order $\ell$
  on a polyhedral domain in $\bbR^d$ ($0<\ell<d$). One of the main tools for the
  analysis is a recently introduced smoothed Poincar\'e lifting operator [M.~Costabel
  and A.~McIntosh, On Bogovski\u{\i} and regularized Poincar\'e integral operators for de
  Rham complexes on Lipschitz domains, Math. Z., (2009)].  In the case
  $\ell=1$ our analysis shows that several widely used families of edge finite
  elements satisfy the discrete compactness property in $p$ and hence provide
  convergent solutions to the Maxwell eigenvalue problem. In particular,
  N\'ed\'elec elements on triangles and tetrahedra (first and second kind) and on
  parallelograms and parallelepipeds (first kind) are covered by our theory.
\end{abstract}

\section{Introduction: Maxwell eigenvalue problem}
\label{sec:introduction}

Maxwell's eigenvalue problem in a closed cavity $\Dom\in\bbR^3$ 
with perfectly conducting walls
can be written as follows by means of the Maxwell-Amp\`ere
and Faraday laws: Find the resonance
frequencies $\omega\in\bbR$ and the electromagnetic
fields $(\E,\H)\not=(0,0)$ such that
\begin{equation}
\label{eq:maxwell}
\begin{array}{cccl}
\curl\E=i\omega\mu\H\quad &\text{ and }\quad
&\curl\H=-i\omega\epsilon\E&\text{ in }\Dom\\
\E\times\bfn=0 \quad &\text{ and }\quad
&\H\cdot\bfn=0\quad &\text{ on }\partial\Dom,
\end{array}
\end{equation}
where $\epsilon$ and $\mu$ denote the dielectric permittivity and magnetic permeability,
respectively. The fields $\E$ and $\H$ are sought in $\Ltwo^3$.

For simplicity, we consider the case of homogeneous isotropic material with
normalized material constants ($\epsilon,\,\mu=1$) --- we will come back to the
general setting in Remark \ref{rem:epsilonmu}. In a classical way, the elimination of the magnetic field from equations \eqref{eq:maxwell} yields  the Maxwell eigenvalue problem with perfectly electrically conducting (PEC) walls in variational form:
\begin{equation}
\label{eq:maxevp}
\aligned
   &\mbox{Seek $\bfu\in\zbHcurl\setminus\{0\}$, \ $\omega\in\R^+_0$ \ such that}\\
   &\SPLtwo{\curl\bfu}{\curl\bfv} = 
   \omega^{2}\SPLtwo{\bfu}{\bfv}\ \forall
   \bfv\in\zbHcurl\;.\endaligned
\end{equation}
The elimination of the electric field would correspond to the same problem 
modelled through replacing $\zbHcurl$ with $\Hcurl$\footnote{By and large, we
  adopt the standard notations for Sobolev spaces, see \cite[Ch.~2]{GIR86}.}.

One aim of this paper is to prove the convergence of 
$H(\curl)$-conforming Galerkin discretizations of Maxwell eigenvalue problem \eqref{eq:maxevp} in the framework of the $p$-version of the finite element method. The finite element
approximation of Maxwell eigenvalues has been the object of intense investigations 
for more than 20 years.  It was soon recognized that the $H(\curl)$-conforming  Galerkin finite element discretizations need special finite element spaces that are generally termed \emph{edge} finite elements (see~\cite{NED80,NED86,BOS88a}).

The first attempts to analyze the discretized eigenvalue problem have been made for the
$h$-version of edge finite elements.  We mention \cite{KIK89} as a pioneering work on
lowest order edge finite elements, where the \emph{discrete compactness property}
(see~\cite{ANS71}) has been indicated as a key ingredient for the analysis. 
Other
relevant works are~\cite{BFG97,BOF99,CFR00,MOD99,KIK01,COD03,BO07}, and we
refer the interested reader to~\cite{HIP02,MON03} and to the references therein
for a review on this topic. 

In these references, the Maxwell eigenvalue problem is often studied using
variational formulations different from \eqref{eq:maxevp}, for example mixed
formulations \cite{BO07}, regularized formulations \cite{CostabelDauge02b,
  CostabelDaugeSchwab05} or mixed regularized formulations \cite{AFW09, BuCJ09}. With
the exception of the method of weighted regularization \cite{CostabelDauge02b,
  CostabelDaugeSchwab05, BuCJ09}, where $H^1$-conforming elements can be used, these
formulations use the $H(\curl)$-conforming edge elements. In their analysis, special
conditions implying convergence of the discrete eigenvalue problems are presented,
for example the so-called Fortid property \cite{BOF99}, or the GAP property
\cite{BUF05}. As explained there, these conditions are related to the discrete
compactness property. Here we choose to work with the simple variational formulation
\eqref{eq:maxevp} and its generalization to differential forms. The role of the
discrete compactness property in this context has been discussed in detail in
\cite{CFR00}.

The analysis presented in the references above covers the
$h$-version for basically all known families of edge finite elements.
It soon turned out, however, that the analysis of the $p$- and $hp$-versions
of edge finite elements needed tools different from those developed for the
$h$-version. In~\cite{BDC03} the two-dimensional triangular case has been studied for
the $hp$-version, but the analysis depends on a conjectured estimate which has only
been demonstrated numerically. In~\cite{BCD06} a rigorous proof for the $hp$-version
of 2D rectangular edge elements has been proposed (allowing for one-irregular hanging
nodes) which, in particular, contains the first proof of eigenvalue/eigenfunction
convergence for the pure spectral method ($p$-version with one element) on a
rectangle.

What paved the way for a successful attack on a general $p$-version analysis was the
regularized Poincar\'e lifting recently introduced in \cite{CMI08}: it enjoys
excellent continuity properties and at the same time respects discrete differential
forms. In this paper we are going to show how the regularized Poincar\'e lifting can be
combined with another recent invention, the projection based interpolation operators,
see \cite{DEM06,DEB04}, to clinch the analysis of the $p$-version of edge elements.
This allows to prove the discrete compactness (and hence the convergence of the
discrete eigensolutions) for a wide class of finite elements related to discrete
differential forms: for \eqref{eq:maxevp} this includes, in particular, N\'ed\'elec
elements on triangles and tetrahedra (first and second kind) and on parallelograms
and parallelepipeds (first kind).

As already mentioned, one of the key ingredients for the convergence analysis is the
discrete compactness property. Much insight can be gained from investigating it in
the more general framework of discrete differential forms (see~\cite{AFW06} for a
lucid introduction to this subject). In this setting, the proofs are more natural and
simultaneously cover, in particular, two- and three-dimensional Maxwell eigenvalue problems.

\medskip\noindent\textbf{Plan of the paper.} \ 
The structure of the paper is as follows. We start in Section~\ref{sec:diffform}
with a generalization of \eqref{eq:maxevp} to eigenvalue problems associated 
with the de Rham complex on differential forms.
Then we define the discrete compactness property and discuss its
significance in the context of Galerkin discretization: in association with
two standard completeness properties, it gives a crucial
sufficient condition for the convergence of eigenvalues and eigenvectors.
Section~\ref{sec:generic-assumptions} is the core of our paper and contains the
description of our abstract assumptions.  Having in mind the $p$-version of
finite elements, we consider a \emph{fixed} mesh $\mesh$ of a bounded Lipschitz polyhedron
$\Dom\subset\bbR^d$ and a sequence of spaces of discrete differential forms of
order $\ell$ (with $0<\ell<d$) together with projection operators onto discrete spaces; 
we prove that our assumptions
imply the validity of the discrete compactness property for such a sequence 
of spaces (Theorem~\ref{thm:main}).
The abstract theory relies on the existence of suitable Poincar\'e lifting
operators which are presented in Section~\ref{sec:smooth-poic-lift}. 
The mapping properties of these lifting operators allow to specify some of the function spaces appearing in our abstract assumptions. In Section~\ref{sec:discr-diff-forms-2} we recall
the classical families of discrete differential forms with high degree polynomial
coefficients on simplicial or tensor product elements.

Our abstract theory applies to any dimension $d$, but for want of suitable
  regularity results, embeddings, and projection operators, we can give examples satisfying all of its
  assumptions only in dimensions $d=2$ and $d=3$. This is done in
  Section~\ref{sec:discr-diff-forms}, where we concretize the function spaces and
recall embedding results and properties of projection based interpolation operators
related to these spaces.  All abstract assumptions are then satisfied, leading to the
main convergence result stated in Theorem \ref{thm:mainMax}. The analysis of a
$p$-version edge element discretization of the Maxwell eigenvalue problem
\eqref{eq:maxevp} is covered as case $d=3$ and $\ell=1$, see Corollary~\ref{cor:mainMax}.

\section{Differential forms and generalized Maxwell eigenvalue problem}
\label{sec:diffform}

The variational eigenvalue problem \eqref{eq:maxevp} turns out to be a
member of a larger family of eigenvalue problems, when viewed from the
perspective of differential forms. This more general perspective
offers the benefit of a unified theoretical treatment of different
kinds of eigenvalue problems, e.g., the scalar Laplace eigenproblem,
Maxwell cavity eigenproblems in dimensions $2$ and $3$, the
eigenproblem for the $\grad\Div$-operator in dimension $3$. This
policy has had remarkable success in numerical analysis recently, \textit{cf.}
\cite{Arnold02}. Thus, in this section we first recall some basic
notions related to differential forms. We refer the interested reader
to~\cite[Sect.~2]{AFW06} for an introduction to this subject.

\subsection{Function spaces of differential forms}
Given a bounded Lipschitz domain $\Omega\subset\bbR^d$, we denote by $\FCinf{\ell}$, $0\leq\ell\leq d$, the
space of smooth differential forms of degree $\ell$ on $\Dom$ and 
by
$\extd[\ell]:\FCinf{\ell}\to\FCinf{\ell+1}$
the \emph{exterior derivative}.

We rely on the Hilbert spaces
\begin{gather}
  \label{def:FHd}
  \FHd{\ell} := \{\bfv\in \FLtwo{\ell}: \extd[\ell]\bfv\in\FLtwo{\ell+1}\}\;,
\end{gather}
where $\FLtwo{\ell}$ is the space of differential $\ell$-forms on $\Dom$ with
square integrable coefficients in their canonical basis representation,
see~\cite[Sect.~2]{CMI08}. Its inner product can be expressed as
\begin{gather}
  \label{eq:SPFLtwo}
  \SPN{\bfu}{\bfv} := \int\nolimits_{\Dom}\bfu\wedge\hodge\bfv\;,\quad
  \bfu,\bfv\in \FLtwo{\ell}\;,
\end{gather}
with $\hodge$ the Hodge star operator induced by the Euclidean metric
on $\bbR^{d}$, which maps $\ell$-forms to $(d-\ell)$-forms. As above,
a $\circ$ tags the subspaces of forms with vanishing trace
$\trace{\partial \Omega}$ on $\partial\Dom$, which can also be
obtained by the completion of compactly supported smooth $\ell$-forms
with respect to the $\FHd{\ell}$-norm:
\begin{equation}
\label{eq:2.3}
   \zbFHd{\ell}  := \{\bfv\in\FHd{\ell}: \trace{\partial \Omega}\bfv = 0 \}.
\end{equation}
The subspace of \emph{closed forms} is the kernel of $\extd[\ell]$ and is denoted by $\zbkFHd{\ell}$:
\begin{equation}
\label{eq:2.4}
   \zbkFHd{\ell}  := \{\bfv\in \zbFHd{\ell}: \extd[\ell]\bfv = 0 \}.
\end{equation}

\subsection{Variational eigenvalue problems}
\label{sec:vari-eigenv-probl}

After choosing bases for the spaces of alternating multilinear forms on $\bbR^{d}$,
vector fields (``vector proxies'') $\Omega\mapsto\bbR^{\binom{d}{\ell}}$ provide
an isomorphic model for differential $\ell$-forms on $\Omega$. Choosing the standard
``Euclidean basis'', the operators $\hodge, \delop, \trace{\partial \Omega}$ are
incarnated by familiar operators of classical vector analysis, different for
different dimension $d$ and degree $\ell$, see Table~\ref{tb:proxies} and
\cite[Table~2.1]{AFW06}.

\begin{table}[ht]
\caption{Identification between (operators on) differential forms and 
(operators on) Euclidean vector proxies in $\bbR^2$ and $\R^3$}
\begin{center}
\begin{tabular}{|l|l|l|l|}
\hline
\multicolumn{2}{|c|}{Differential form}&
\multicolumn{2}{c|}{Proxy representation}\\
\hline
\multicolumn{2}{|c|}{}&\multicolumn{1}{c|}{$d=2$}&\multicolumn{1}{c|}{$d=3$}\\
\hline \hline
\multirow{3}{*}{$\ell=0$}
&$\extd[0]$&$\grad$&$\grad$\\
&$\trace{\partial\Dom}\phi$&$\phi_{|\partial\Dom}$
&$\phi_{|\partial\Dom}$\\
&$\zbFHd{0}$&$\zbHone$&$\zbHone$\\
\hline
\hline
\multirow{4}{*}{$\ell=1$}
&$\extd[1]$&$\curltd$&$\curl$\\
&$\trace{\partial\Dom}\bfu$&$(\bfu\times\bfn)_{|\partial\Dom}$
&$(\bfu\times\bfn)_{|\partial\Dom}$\\
&$\zbFHd{1}$&$\zbHcurltd$&$\zbHcurl$\\
\hline
\hline
\multirow{4}{*}{$\ell=2$}
&$\extd[2]$&$0$&$\div$\\
&$\trace{\partial\Omega}\bfq$&$0$&$(\bfq\cdot\bfn)_{|\partial\Omega}$\\
&$\zbFHd{2}$&$\Ltwo$&$\zbHdiv$\\
\hline
\end{tabular}
\end{center}
\label{tb:proxies}
\end{table}

Hence, the eigenvalue problem \eqref{eq:maxevp} with
$\epsilon,\mu\equiv 1$ is the special case $d=3$, $\ell=1$, of the
following variational eigenvalue problem for differential
$\ell$-forms, $0\leq \ell < d$: 
\begin{equation}
\label{eq:evp}
\aligned
  &\mbox{Seek $\bfu\in\zbFHd{\ell}\setminus\{0\}$, \ $\omega\in\bbR^{+}_{0}$, such that}\\
  &\SPN{\extd[\ell]\bfu}{\extd[\ell]\bfv} = \omega^{2}\SPN{\bfu}{\bfv}\quad
  \forall \bfv\in\zbFHd{\ell}\;.\\[1ex]
\endaligned
\end{equation}
A key observation is that the bilinear form $(\bfu,\bfv)\mapsto
\SPN{\extd[\ell]\bfu}{\extd[\ell]\bfv}$ has an infinite dimensional
kernel $\zbkFHd{\ell}$ comprising all closed $\ell$-forms. It provides
the invariant subspace associated with the essential spectrum $\{0\}$
of \eqref{eq:evp}. This essential spectrum can be identified as the
main source of difficulties confronted in the Galerkin discretization
of \eqref{eq:evp}.

On the other hand, any solution $\bfu$ of \eqref{eq:evp} for
$\omega\not=0$ satisfies $\SPN{\bfu}{\extd[\ell-1]\psibf} = 0$ for all
$\psibf\in\zbFHd{\ell-1}$. Thus the eigenfunctions corresponding to
non-zero eigenvalues belong to the subspace
\begin{gather}
\label{eq:zbY}
   \zbFXd{\ell} := \{ \bfv\in\zbFHd{\ell}: \SPN{\bfv}{\extd[\ell-1]\psibf} = 0 
   \;\;\forall \psibf\in\zbFHd{\ell-1}  \},
\end{gather}
which means they belong to the kernel of $\delop[\ell]$. This is the
generalization of the divergence free constraint found for electric
fields in the Maxwell case. From \cite{PIC84} we learn the following
theorem.

\begin{theorem}
  \label{thm:compemb}
  For any $d\in\bbN$, $0\leq l \leq d$, the embedding
  of $\zbFXd{\ell}$ in $\FLtwo{\ell}$
  is compact. 
\end{theorem}

Thus, by restricting the eigenvalue problem to $\zbFXd{\ell}$, we can use 
Riesz-Schauder theory. This implies that \eqref{eq:evp} gives rise to an unbounded
sequence of positive eigenvalues $\lambda^k=(\omega^k)^2$
\begin{gather}
\label{eq:eval}
 \lambda^0=0 < \lambda^1 \le \lambda^2\le\dots\,,
 \qquad \lambda^k\to\infty \; (k\to\infty)\,,
\end{gather}
with associated finite dimensional mutually $\Ltwo$-orthogonal eigenspaces.

\begin{remark}
  \label{rem:NBC}
  Owing to the zero trace boundary conditions imposed on the functions in
  \eqref{eq:evp}, it may be called a Dirichlet eigenvalue problem. Using
  $\FHd{\ell}$ as variational space would result in the corresponding Neumann
  eigenvalue problem. Its analysis runs utterly parallel to the Dirichlet
  case using the techniques presented below. 
\end{remark}

\subsection{Approximation of the eigenvalue problem and the role of discrete compactness}
\label{sec:discrete-compactness}

In the sequel we fix the degree $\ell$, $0\leq\ell< d$, of the
differential forms. Spaces of \emph{discrete differential forms}
\begin{gather*}
  \zbWpl[]{\ell}{p}\subset\zbFHd{\ell}\;,\quad \dim  \zbWpl[]{\ell}{p} < \infty\;,
\end{gather*}
lend themselves to a straightforward discretization of \eqref{eq:evp}. In this section,
$p\in\bbN$ stands for an abstract discretization parameter, and, sloppily speaking,
large values of $p$ hint at trial/test spaces of high resolution. 

Consider the approximation of the eigenvalue problem \eqref{eq:evp}
by the Galerkin method: 
\begin{equation}
  \label{eq:evp_p}
\aligned
  &\mbox{Find $\bfu_p\in\zbWpl[]{\ell}{p}\setminus\{0\}$, \ $\omega\in\bbR^{+}_{0}$, such that}\\
  &\SPN{\extd[\ell]\bfu_p}{\extd[\ell]\bfv_p} = \omega^{2}\SPN{\bfu_p}{\bfv_p}\quad
  \forall \bfv\in\zbWpl[]{\ell}{p}\;.\\[1ex]
\endaligned
\end{equation}
Now, the key issue is convergence of eigenvalues and eigenvectors as
$p\to\infty$, rigorously cast into the concept of \emph{spectrally
  correct, spurious-free approximation} \cite[Sect.~4]{CFR00}. Let us recall
these notions in a few words for the case of self-adjoint nonnegative
operators without continuous spectrum (which is the case here).

The spectral correctness of the approximation of an eigenvalue problem
such as \eqref{eq:evp} by a sequence of finite rank eigenvalue
problems \eqref{eq:evp_p} means that all eigenvalues and all
eigenvectors of \eqref{eq:evp} are approached by the eigenvalues and
eigenvectors of \eqref{eq:evp_p} as $p\to\infty$. If \eqref{eq:evp}
has a compact resolvent (which is the case \emph{only when} $\ell=0$),
the spectral correctness is an optimal notion: It implies that if
$\{\lambda^k\}_{k\ge1}$ and $\{\lambda^k_p\}_{k\ge1}$ are the
increasing eigenvalue sequences of \eqref{eq:evp} and \eqref{eq:evp_p}
(with eigenvalues repeated according to their multiplicities), then
\begin{equation}
\label{eq:conveig}
   \lambda^k_p \to \lambda^k\quad\mbox{as}\quad p\to\infty\quad \forall k\ge1,
\end{equation}
and the gaps between eigenspaces (correctly assembled according to
multiplicities of the eigenvalues of \eqref{eq:evp}) tend to $0$ as
$p\to\infty$.

If we face an eigenvalue problem for a self-adjoint non-negative
operator with an infinite dimensional kernel, and otherwise discrete
positive spectrum (which is the case for \eqref{eq:evp} for all
$\ell\ge1$), the spectral correctness implies the same properties as
above with the following modifications of the definitions: Now
$\{\lambda^k\}_{k\ge1}$ is the increasing sequence of \emph{positive}
eigenvalues of \eqref{eq:evp} (as specified in \eqref{eq:eval}) and,
given a positive number $\varepsilon<\lambda^1$,
$\{\lambda^k_p\}_{k\ge1}$ is the increasing sequence of the
eigenvalues of \eqref{eq:evp_p} larger than $\varepsilon$ (still with
repetitions according to multiplicities). With such conventions, 
\emph{spectral correctness} still implies convergence of eigenvalues
\eqref{eq:conveig} and eigenspaces as above. In this context, 
\emph{spurious-free approximation} means that there exists
$\varepsilon_0>0$ such that all eigenvalues of \eqref{eq:evp_p} less
than $\varepsilon_0$ are zero. Therefore, spectrally correct,
spurious-free approximation implies the convergence property
\eqref{eq:conveig} and the corresponding convergence of eigenspaces, if
we define $\{\lambda^k_p\}_{k\ge1}$ as the increasing sequence of the
\emph{positive eigenvalues of} \eqref{eq:evp_p}.

There exist several different ways, all well studied and summarized in
the literature of the last decade, for proving the convergence of the
discrete eigenvalue problem \eqref{eq:evp_p} to the continuous
eigenvalue problem \eqref{eq:evp}: One can use a reformulation as an
eigenvalue problem in mixed form as analyzed in \cite{BO07}, or one
can use a regularization which gives an elliptic eigenvalue problem
for the Hodge-Laplace operator as analyzed in \cite{AFW06}, or one can
follow the arguments of \cite{CFR00} and study the non-elliptic
problem \eqref{eq:evp} directly.

Here we outline the latter approach, which employs the analysis of
\cite{DNR78a} of the approximation of eigenvalue problems of
non-compact selfadjoint operators. Since \cite{CFR00} deals only with
the Maxwell case, i.\- e. $d=3$, $\ell=1$, we examine the main
arguments, in order to verify that they are also valid for the general
case. The proofs we give are adaptations of those of \cite{CFR00} to our more general situation.

Let us define the solution operator $A:\FLtwo{\ell}\to\zbFHd{\ell}$ of the source problem corresponding to the eigenvalue problem \eqref{eq:evp} and its discrete counterpart
 $A_p:\zbWpl[]{\ell}{p}\to\zbWpl[]{\ell}{p}$ by
 \begin{gather}
 \label{eq:A&A_p}
 \begin{aligned}
 \SPN{\extd[\ell]A\bff}{\extd[\ell]\bfv} + \SPN{A\bff}{\bfv} &= \SPN{\bff}{\bfv}
 \quad \forall \bfv\in \zbFHd{\ell}
 \\
 \SPN{\extd[\ell]A_p\bff}{\extd[\ell]\bfv} + \SPN{A_p\bff}{\bfv} &= \SPN{\bff}{\bfv}
 \quad \forall \bfv\in \zbWpl[]{\ell}{p}\,.
 \end{aligned}
\end{gather}

Note that the operators $A$ and $A_p$ have the same eigenfunctions and the same eigenvalues (after a transformation) as the eigenvalue problems \eqref{eq:evp} and \eqref{eq:evp_p}. Namely, \eqref{eq:evp} and \eqref{eq:evp_p} are equivalent to the relations
\begin{gather}
\label{eq:evpA}
  \bfu=(\omega^2+1)A\bfu\,;\qquad
  \bfu_p=(\omega^2+1)A_p\bfu_p\,.
\end{gather}
The infinite-dimensional eigenspace at $\omega=0$ shows that $A$ is not a compact operator.

Following \cite{CFR00}, three conditions are identified that together are necessary and sufficient for a spectrally correct, spurious-free approximation of $A$ by $A_p$ or, equivalently, of the eigenvalue problem \eqref{eq:evp} by the discrete eigenvalue problem \eqref{eq:evp_p}. 

The first condition is rather natural. It states that the sequence  of discrete spaces
$\big(\zbWpl[]{\ell}{p}\big)_{p\in\bbN}$ is asymptotically dense in $\zbFHd{\ell}$
(compare \cite[Condition (CAS) -- completeness of approximating subspaces]{CFR00})
\begin{equation}
  \label{eq:CAS}
  \mbox{} \hskip -0.9em\mbox{(CAS)} \hskip 3em
  \lim_{p\to\infty}\inf_{\bfv_p\in\zbWpl[]{\ell}{p}}\NFHd{\ell}{\bfv-\bfv_p}=0\quad
  \forall\bfv\in\zbFHd{\ell}\;.
   \hskip 3em
\end{equation}
The second condition, only relevant for $\ell>0$, states that closed forms can be well
approximated by discrete closed forms (compare \cite[Condition (CDK) -- completeness of discrete kernels]{CFR00})
\begin{equation}
  \label{eq:CDK}
  \mbox{} \hskip -0.9em\mbox{(CDK)} \hskip 1em
  \lim_{p\to\infty}\inf_{\bfz_p\in\zbWpl[]{\ell}{p}\cap\zbkFHd{\ell}}
  \NLtwo{\bfz-\bfz_{p}} = 0\quad
  \forall\bfz\in\zbkFHd{\ell}\;.
   \hskip 1em
\end{equation}
The third condition is the most intricate one and has been dubbed
\emph{discrete compactness}. For its formulation, we introduce the orthogonal
complement space of the discrete closed forms:
\begin{gather}
  \label{eq:Zpldef}
  \zbZpl[]{\ell}{p} := \{\bfu_{p}\in \zbWpl[]{\ell}{p}:\;
  \SPN{\bfu_{p}}{\bfz_{p}} = 0 \quad
       \forall \bfz_p\in\zbWpl[]{\ell}{p}\cap\zbkFHd{\ell}\}.
\end{gather}  

\begin{definition}
  \label{def:dc}
  Let us choose $\ell\in\{1,\ldots,d-1\}$.
  The \emph{discrete compactness property} holds for a family
  $\big(\zbWpl[]{\ell}{p}\,\big)_{p\in\mathbb{N}}$ of finite dimensional subspaces of
  $\zbFHd{\ell}$, if for any subsequence $\N'$ of $\N$, any \emph{bounded} sequence 
$$
  \big(\bfu_{p}\big)_{p\in\mathbb{N'}}
  \subset\zbFHd{\ell}\quad\mbox{with}\quad\bfu_{p}\in\zbZpl[]{\ell}{p}
$$ 
contains a subsequence that
\emph{converges in $\FLtwo{\ell}$}.
\end{definition}

The convergence proof is based on two lemmas, the first of which corresponds to \cite[Theorem 4.12]{CFR00}. It implies, according to \cite[Condition P1) and Theorems 2,4,5,6]{DNR78a}, the spectral correctness of the approximation.

\begin{lemma}
\label{lem:CAS+DCP->CHN}
 If \eqref{eq:CAS} and the discrete compactness property hold, then 
\begin{gather}
 \label{eq:CHN}
  \lim_{p\to\infty} \sup_{\bfv_p\in\zbWpl[]{\ell}{p}\,;\, \NFHd{\ell}{\bfv_p}=1}
  \NFHd{\ell}{A\bfv_p-A_p\bfv_p}=0\,.
\end{gather}
\end{lemma}
\begin{proof}
Note first that for $\bfv_p\in\zbWpl[]{\ell}{p}\cap\zbkFHd{\ell}$ there holds
$A\bfv_p=\bfv_p=A_p\bfv_p$, so that by orthogonal decomposition of $\zbWpl[]{\ell}{p}$ one gets
\[
 \sup_{\bfv_p\in\zbWpl[]{\ell}{p}\,;\, \NFHd{\ell}{\bfv_p}=1} \hskip-9pt
  \NFHd{\ell}{A\bfv_p-A_p\bfv_p}
 \ = \!\!\!\!
 \sup_{\bfv_p\in\zbZpl[]{\ell}{p}\,;\, \NFHd{\ell}{\bfv_p}=1} \hskip-9pt
  \NFHd{\ell}{A\bfv_p-A_p\bfv_p}\,.
\]
Furthermore, one has by definition of $A$ and $A_p$
\[
 \NFHd{\ell}{A\bfv_p-A_p\bfv_p} = 
   \inf_{\bfw_p\in\zbWpl[]{\ell}{p}} \NFHd{\ell}{A\bfv_p-\bfw_p}\,.
\]
Assume now that \eqref{eq:CHN} does not hold. 
Then there exists $\varepsilon>0$, a subsequence $\N'$ of $\N$ and a sequence $(\bfv_p)_{p\in\N'}$ with 
$\bfv_p\in\zbZpl[]{\ell}{p}$ satisfying $\NFHd{\ell}{\bfv_p}=1$ and 
\begin{gather}
\label{eq:geeps}
  \NFHd{\ell}{A\bfv_p-\bfw_p} \ge \varepsilon 
  \qquad
  \forall p\in\N',\; \bfw_p\in\zbWpl[]{\ell}{p}\,.
\end{gather}
We can apply the discrete compactness property to the sequence $(\bfv_p)$ and obtain a subsequence converging in $\FLtwo{\ell}$ to some $\bfv\in\FLtwo{\ell}$. 
Since $A:\FLtwo{\ell}\to\zbFHd{\ell}$ is continuous, we find $A\bfv\in\zbFHd{\ell}$, and the approximation property \eqref{eq:CAS} provides us with a sequence $(\bfw_p)$ with $\bfw_p\in\zbWpl[]{\ell}{p}$ that converges in $\zbFHd{\ell}$ to $A\bfv$. Hence for the subsequence we obtain 
\[
  \NFHd{\ell}{A\bfv_p-\bfw_p} \le 
  \NFHd{\ell}{A\bfv_p-A\bfv} + \NFHd{\ell}{A\bfv-\bfw_p} \to 0\,,
\]
in contradiction with \eqref{eq:geeps}.
\end{proof}

The second lemma corresponds to \cite[Corollary 2.20]{CFR00}. It gives the discrete Friedrichs inequality (in \cite{BO07} also called ``ellipticity in the discrete kernel''), and it is easy to see that this implies that $\omega=0$ is not a limit point of positive discrete eigenvalues, so that the spurious-free property of the approximation follows.

\begin{lemma}
\label{lem:CDK+DCP->DFI}
  If \eqref{eq:CDK} and the discrete compactness property hold, then there exists $\alpha>0$ such that for all $p\in\N$
  \begin{gather}
  \label{eq:DFI}
   \NLtwo{\extd[\ell]\bfv} \ge \alpha \NLtwo{\bfv}
   \quad \forall \bfv \in \zbZpl[]{\ell}{p}
  \end{gather}
\end{lemma}
\begin{proof}
Assume that \eqref{eq:DFI} does not hold. Then there exists a subsequence $\N'$ of $\N$ and a sequence
$(\bfv_p)_{p\in\N'}$ with $\bfv_p\in\zbZpl[]{\ell}{p}$ satisfying 
\begin{gather}
\label{eq:notDFI}
  \NLtwo{\bfv_p}=1 \quad\mbox{ and }\; 
  \lim_{p\to\infty} \NLtwo{\extd[\ell]\bfv_p} =0\,.
\end{gather}
The discrete compactness property can be applied to this sequence and gives a subsequence converging in $\FLtwo{\ell}$ to some $\bfz\in\FLtwo{\ell}$. 
From \eqref{eq:notDFI} follows that the convergence actually takes place in $\zbFHd{\ell}$ and that $\bfz\in\zbkFHd{\ell}$. 
Therefore the approximation property \eqref{eq:CDK} provides us with a sequence $(\bfz_p)$ with $\bfz_p\in\zbWpl[]{\ell}{p}\cap\zbkFHd{\ell}$ 
that converges in $\FLtwo{\ell}$ to $\bfz$. Hence for the subsequence we find
\[
  \NLtwo{\bfv_p-\bfz_p}\le \NLtwo{\bfv_p-\bfz} + \NLtwo{\bfz-\bfz_p} \to 0\,.
\]
But $\bfv_p\in\zbZpl[]{\ell}{p}$ and $\bfz_p\in\zbWpl[]{\ell}{p}\cap\zbkFHd{\ell}$  are $\Ltwo$-orthogonal, hence for all $p$
\[
  \NLtwo{\bfv_p-\bfz_p}^2 = \NLtwo{\bfv_p}^2 + \NLtwo{\bfz_p}^2 \ge 1\,,
\]
which leads to a contradiction.
\end{proof}

To summarize, Lemmas \ref{lem:CAS+DCP->CHN} and \ref{lem:CDK+DCP->DFI} together prove the following result. 

\begin{theorem}
\label{thm:spurfree}
If the completeness of approximating subspaces \eqref{eq:CAS}, the completeness of discrete kernels \eqref{eq:CDK} and the discrete compactness property hold, then \eqref{eq:evp_p} provides a spectrally correct, spurious-free approximation of the eigenvalue problem \eqref{eq:evp}.
\end{theorem}

\begin{remark}
The main focus of this section is on the \emph{convergence} of the eigenvalues and
the eigenfunctions of problem~\eqref{eq:evp_p} to those
of~\eqref{eq:evp}. On the other hand, when considering concrete applications
it is crucial to investigate the \emph{order} of convergence. In order to do so,
several strategies are available. A straightforward approach which well fits
the theory summarized in this section makes use of the results from
\cite{DNR78b}. Theorem~1 of~\cite{DNR78b} states in this particular situation
that the error in the eigenfunctions (measured as usual by the gap of Hilbert
spaces) is bounded by the best approximation, and Theorem~3(c) of~\cite{DNR78b}
states that the eigenvalues achieve double order of convergence
since our problem is symmetric. An alternative approach makes use of the
equivalence of problems~\eqref{eq:evp} and~\eqref{eq:evp_p} with suitable
mixed formulations~\cite[Part 4]{BOF10}; in this case an estimate of the
order of convergence can be achieved by the standard Babu\v ska--Osborn
theory for the spectral approximation of compact operators applied to the
mixed formulations~\cite[Theorems 13.8, 13.10, 14.9, 14.11]{BOF10}.
\end{remark}

\section{Abstract framework implying discrete compactness}
\label{sec:generic-assumptions}

In this section we fix a degree of differential forms
\[
   \ell\in\{1,\ldots,d-1\},
\]
and we formulate a set of hypotheses which allow us to prove the
discrete compactness property. These hypotheses are organized in three
groups:
\begin{enumerate}
\item standard assumptions related to the finite element spaces $\zbWpl[]{\ell}{p}$
  (Sect.~\ref{sec:discr-diff-forms-1}),
\item assumptions on the existence and key properties of ``lifting operators''
  (Sect.~\ref{sec:local-lifting}),
\item hypotheses on projections onto $\zbWpl[]{\ell}{p}$ complying
  with the commuting diagram property and satisfying an approximation
  property (Sect.~\ref{sec:local-projectors}).
\end{enumerate}
To state these assumptions we have to introduce intermediate spaces $X$ and $S$ of more regular forms
\[
\zbWpl[]{\ell}{p}\subset\HX{\ell}\subset\zbFHd{\ell}
\quad\mbox{and}\quad
\zbWpl[]{\ell-1}{p}\subset\HS{\ell-1}\subset\zbFHd{\ell-1}\;,
\]
allowing compact embedding arguments and precise notions of continuity of lifting and
projection operators. 

\subsection{Discrete spaces}
\label{sec:discr-diff-forms-1}

Our focus is on finite element spaces. For the sake of simplicity, we restrict
ourselves to polyhedral Lipschitz domains $\Dom$. We assume that the finite
dimensional trial and test spaces $\zbWpl[]{\ell}{p}$, $p\in\bbN$, are based on a
\emph{fixed} finite partition $\mesh$ of $\Dom$, composed of elements (cells) $K$:
\begin{equation*}
  \overline{\Dom} = \bigcup\limits_{K\in\mesh}\overline{K}\quad,\quad
  K\cap K'=\emptyset\;\text{, if }K\not=K',\;K,K'\in\mesh\;.
\end{equation*}%

For a cell $K\in\mesh$, let $\Faces[K]{m}$ designate the set of $m$-dimensional
facets of $K$: for $m=0$ these are the vertices, for $m=1$ the edges, for $m=d-1$ the
faces, and $\Faces[K]{d} = \{K\}$.

We take for granted that the discrete
spaces $\zbWpl[]{\ell}{p}$ can be assembled from local contributions in the sense
that for each mesh cell $K\in\mesh$ there is a space $\Wpl[K]{\ell}{p}\subset
\FCinf[\overline{K}]{\ell}$ of smooth $\ell$-forms on $K$, such that
\begin{equation}
  \label{eq:Wlploc}
  \zbWpl[]{\ell}{p} = \zbWpl[\mesh]{\ell}{p} 
  := \big\{\bfv\in\zbFHd{\ell}:\,\bfv\on{K} \in \Wpl[K]{\ell}{p}\;\forall K\in\mesh\big\}\;.
\end{equation}
In other words, $\zbWpl[]{\ell}{p}$ can be defined by specifying the local spaces 
$\Wpl[K]{\ell}{p}$ and requiring the continuity of traces across inter-element
boundaries as well as boundary conditions on $\partial\Dom$.

In the same fashion, we introduce a corresponding family
$\zbWpl[]{\ell-1}{p}\subset \zbFHd{\ell-1}$ of spaces of discrete
$(\ell-1)$-forms. We will see later on that as a consequence of further
hypotheses, the local spaces $\Wpl[K]{\ell-1}{p}$ and
$\Wpl[K]{\ell}{p}$ satisfy an exact sequence property.

\subsection{Spaces of more regular forms}
\label{sec:regularity}

We introduce a Hilbert space $\HX{\ell}\subset\zbFHd{\ell}$ that
captures the extra regularity that distinguishes $\ell$-forms in the
space $\zbFXd{\ell}$. We can think of this space as a space of ``more
regular'' $\ell$-forms on $\Dom$.

\ass{ass:reg}{\ \\
The space $\zbFXd{\ell}$  defined in \eqref{eq:zbY} is continuously embedded in $\HX{\ell}$.
}
\noindent
This means that with $C>0$ depending only on $\Dom$
\begin{equation}
  \label{eq:reg}
  \NHX{\ell}{\bfu} \leq C \NFHd{\ell}{\bfu}\quad\forall 
  \bfu \in \zbFXd{\ell}\;.
\end{equation}
On the other hand, $\HX{\ell}$ has to be small enough to maintain the compact
embedding satisfied by $\zbFXd{\ell}$, \textit{cf.} Thm.~\ref{thm:compemb}.

\ass{ass:compemb}{
The space $\HX{\ell}$ is compactly embedded in $\FLtwo{\ell}$.
}

As with the discrete spaces, the spaces $\HX{\ell}$ are built from local
contributions and will therefore depend on the mesh $\mesh$. We assume that for
each mesh cell $K\in\mesh$ there are Hilbert spaces $\HX[K]{\ell}$ so that:
\begin{equation}
  \label{eq:XK}
  \HX{\ell} =
  \big\{\bfv\in\zbFHd{\ell}:\,\bfv\on{K} \in \HX[K]{\ell}\ \ \forall K\in\mesh\big\}\,,
\end{equation}
and, in addition, the norm of $\HX{\ell}$ is defined through local
contributions:
\begin{equation}
  \label{eq:NHXloc}\mbox{}\hspace{-0.5em}
  \NHX{\ell}{\bfu}^{2} = \NFHd{\ell}{\bfu}^{2} + \sum\limits_{K\in\mesh}
  \NHX[K]{\ell}{\bfu\on{K}}^{2}\;.
\end{equation}
Finally, the local spaces have to be large enough to contain the discrete forms
for any value of $p$:
\begin{equation}
  \label{eq:Xdisc}
  \Wpl[K]{\ell}{p} \subset 
  \HX[K]{\ell}\;.
\end{equation}

Correspondingly, we introduce a space $\HS{\ell-1} \subset \zbFHd{\ell-1}$ of
``more regular potentials''.  Similar to $\HX{\ell}$, the spaces $\HS{\ell-1}$
are mesh-depen\-dent and allow for a characterization through local Hilbert spaces
$\HS[K]{\ell-1}$, $K\in\mesh$,
\begin{equation}
  \label{eq:SK}
  \HS{\ell-1} =
  \big\{\psibf\in\zbFHd{\ell-1}:\,
  \psibf\on{K} \in \HS[K]{\ell-1}\ \ \forall K\in\mesh\big\}\;.
\end{equation}
They are endowed with the norm
\begin{equation}
  \label{eq:NHSloc}
  \NHS{\ell-1}{\phibf}^{2} = \NFHd{\ell-1}{\phibf}^{2} + \sum\limits_{K\in\mesh}
  \NHS[K]{\ell-1}{\phibf\on{K}}^{2}\;.
\end{equation}
The local spaces are large enough to contain the local discrete potential
spaces:
\begin{equation}
  \label{eq:Sdisc}
  \Wpl[K]{\ell-1}{p}  \subset 
  \HS[K]{\ell-1}\;.
\end{equation}

\noindent
The following assumption establishes the connection between 
$\HX{\ell}$ and $\HS{\ell-1}$.

\ass{ass:Smax}{%
   The exterior derivative maps $\HS{\ell-1}$ continuously into $\HX{\ell}$:
\begin{equation*}
   \HS{\ell-1} \subset \{\phibf\in\zbFHd{\ell-1}:\extd[\ell-1]\phibf\in\HX{\ell}\},
\end{equation*}
 and the image is maximal:
\begin{equation*}
   \extd[\ell-1] \HS{\ell-1} = \extd[\ell-1] \zbFHd{\ell-1} \cap \HX{\ell}.
\end{equation*}}

To conclude this subsection, note that in the case of an element $K$ touching
the boundary $\partial\Dom$, like for the discrete spaces $\Wpl[K]{\ell}{p}$ and
$\Wpl[K]{\ell-1}{p}$, the local spaces $\HX[K]{\ell}$ and $\HS[K]{\ell-1}$ are
not obliged to comply with any boundary conditions.

\subsection{Local liftings}
\label{sec:local-lifting}

A pair of linear mappings $\Poinc_{k,K}:\FCinf[K]{k}\mapsto\FCinf[K]{k-1}$,
$k=\ell,\,\ell+1$, is called a \emph{lifting operator} of degree $\ell$ if it
fulfills
\begin{equation}
\label{eq:lift}
  \extd[\ell-1]\circ\,\Poinc_{\ell,K}+\Poinc_{\ell+1,K}\circ\extd[\ell] 
  = \Op{Id}_{\ell}\;.
\end{equation}
This relation characterizes a ``contracting homotopy'' of the de Rham complex \cite[Section 5.1.2]{AFW09}.

Besides this algebraic relationship, our approach hinges on smoothing properties
of the lifting operators, expressed by means of the local spaces
$\HS[K]{\ell-1}$ of more regular potentials and $\HX[K]{\ell}$ of more regular
forms.  The next assumption summarizes the continuity expected from the lifting
operator.

\ass{ass:lift}{%
  For every $K\in\mesh$ there is a lifting operator $(\Poinc_{\ell,K},\Poinc_{\ell+1,K})$
  whose components can be extended to continuous
  mappings 
\begin{equation*}
  \Poinc_{\ell+1,K}:\FLtwo[K]{\ell+1}\mapsto \HX[K]{\ell} \quad \mbox{and}\quad
  \Poinc_{\ell,K}:\HX[K]{\ell}\mapsto\HS[K]{\ell-1}\,,
\end{equation*}
and thus identity \eqref{eq:lift} holds on $\HX[K]{\ell}$.
}%

As a consequence, for each cell $K\in\mesh$, we have the exact sequence
\begin{equation}
\label{eq:Sex}
    \begin{CD}
       \HS[K]{\ell-1} @>{\extd[\ell-1]}>> \HX[K]{\ell} @>{\extd[\ell]}>> 
       \FLtwo[K]{\ell+1}.
    \end{CD}
\end{equation}

Finally, the local liftings have to be compatible with the local spaces of discrete
differential forms:

\ass{ass:Wlift}{%
  The local operators $\Poinc_{\ell+1,K}$, when applied to exact local discrete
  $(\ell+1)$-forms, yield local discrete $\ell$-forms, \emph{i.e.},
  \begin{equation*}
    \Poinc_{\ell+1,K}\circ \extd[\ell] : \;\;\Wpl[K]{\ell}{p} \to \Wpl[K]{\ell}{p}\;.
  \end{equation*}}

\subsection{Local projectors}
\label{sec:local-projectors}

As usual in methods based on discrete commuting diagrams we need projection operators
$\LPO{k}{p}$ onto discrete spaces for $(\ell-1)$-forms and $\ell$-forms. For degree
$\ell-1$, our local spaces $\HS[K]{\ell-1}$ of more regular potentials can play the
role of domains for the projectors $\LPO{\ell-1}{p}$. For the degree $\ell$, by
generalization of what we actually need in the case of dimension $d=2$ and $d=3$ for
Maxwell, we define our projectors $\LPO{\ell}{p}$ on smaller spaces than
$\HX[K]{\ell}$. We denote these new spaces by $\HS[K]{\ell}$ and require that they
contain for all $p$ the $p$-dependent subspaces
\begin{equation}
\label{eq:XptildeK}
   \HXt[K]{\ell}{p} = \{\bfu\in \HX[K]{\ell} : \extd[\ell] \bfu \in \extd[\ell]\Wpl[K]{\ell}{p}\}\,.
\end{equation}
On the same model as \eqref{eq:SK}-\eqref{eq:NHSloc}, we define the corresponding global spaces $\HS{\ell}$ and
\begin{equation}
\label{eq:Xptilde}
 \HXt{\ell}{p} = \{\bfu\in \HX{\ell} : \extd[\ell] \bfu \in \extd[\ell]\Wpl[]{\ell}{p}\}
\end{equation}
and we have the continuous embeddings
\begin{equation}
\label{eq:XSX}
   \HXt{\ell}{p}  \contemb \HS{\ell}  \contemb \HX{\ell} \;.
\end{equation}

\ass{ass:lp}{
There are \emph{local} continuous linear projections
\[
  \LPO{\ell-1}{p}:\HS[K]{\ell-1}\mapsto\Wpl[K]{\ell-1}{p}
  \quad\mbox{and}\quad
  \LPO{\ell}{p}:\HS[K]{\ell}\mapsto\Wpl[K]{\ell}{p}
\]
for all mesh cells $K\in\mesh$.
}

The standard commuting diagram property is as follows.

\ass{ass:cdp}{The projectors $\LPO{\ell-1}{p}$ and $\LPO{\ell}{p}$ are compatible
  with the exterior derivative in the sense that the diagram
\begin{equation*}
    \begin{CD}
      \HS[K]{\ell-1} @>{\extd[\ell-1]}>> \HS[K]{\ell} \\
      @V{\LPO[K]{\ell-1}{p}}VV  @VV{\LPO[K]{\ell}{p}}V \\
      \Wpl[K]{\ell-1}{p} @>{\extd[\ell-1]}>> \Wpl[K]{\ell}{p}\;,
    \end{CD}
\end{equation*}
commutes for every $K\in\mesh$.}

Let us note that,
as a consequence of Assumptions \ref{ass:lift} and \ref{ass:cdp},
we find that the sequence
\begin{equation*}
     \begin{CD}
        \Wpl[K]{\ell-1}{p} @>{\extd[\ell-1]}>> \Wpl[K]{\ell}{p} @>{\extd[\ell]}>> 
        \extd[\ell]\big(\Wpl[K]{\ell}{p}\big) 
     \end{CD}
\end{equation*}
is exact.

Besides, the local projections acting on $(\ell-1)$-forms are supposed to enjoy a crucial
approximation property using the Hilbert space norms $\NHS[K]{\ell-1}{\cdot}$.

\ass{ass:locapprox}{%
  There is a function $\varepsilon_{\ell-1}\!:\bbN\mapsto
  \bbR^{+}$ with $\lim\limits_{p\to\infty}\varepsilon_{\ell-1}(p) = 0$ so that
  \begin{equation*}
    \NFLtwo[K]{\ell}{\extd[\ell-1](\phibf-\LPO{\ell-1}{p}\phibf)} 
    \leq \varepsilon_{\ell-1}(p)
    \NHS[K]{\ell-1}{\phibf}\quad\forall \phibf\in \HS[K]{\ell-1}\;.
  \end{equation*}}

Finally we assume for the projections $\LPO{\ell}{p}$ a natural condition of
conformity: For all $\bfu\in \HXt[K]{\ell}{p}$
\begin{equation}
  \label{eq:LPO}
  \trace{F}\bfu = 0 \quad\Rightarrow\quad
  \trace{F}\LPO{\ell}{p}\bfu = 0 \quad
   \forall F\in\Faces[K]{m}\;,\quad \ell\leq m \leq d\,,
\end{equation}
and the corresponding condition for the projections $\LPO{\ell-1}{p}$.
This makes it possible to define \emph{global} linear projections
\[
   \PO{\ell}{p}:\HS{\ell} \mapsto \zbWpl[]{\ell}{p}
   \quad\mbox{and}\quad
   \PO{\ell-1}{p}:\HS{\ell} \mapsto \zbWpl[]{\ell-1}{p}
\]
by patching together the local operators 
\begin{equation}
  \label{eq:POdef}
  (\PO{\ell}{p}\bfu)\on{K} := \LPO{\ell}{p}(\bfu\on{K})
  \quad\mbox{and}\quad
  (\PO{\ell-1}{p}\phibf)\on{K} := \LPO{\ell-1}{p}(\phibf\on{K})
  \quad\forall K\in\mesh\,.
\end{equation}
As a consequence of Assumption \ref{ass:cdp} and \eqref{eq:POdef}, the global projectors
$\PO{\ell-1}{p}$ and $\PO{\ell}{p}$ inherit the global \emph{commuting diagram property}
\begin{equation}
  \label{eq:globalcdp}
  \begin{CD}
    \HS{\ell-1} @>{\extd[\ell-1]}>> \HS{\ell} \\
    @V{\PO{\ell-1}{p}}VV  @VV{\PO{\ell}{p}}V \\
    \zbWpl[]{\ell-1}{p} @>{\extd[\ell-1]}>> \zbWpl[]{\ell}{p}.
  \end{CD}
\end{equation}

\subsection{Proof of the discrete compactness property}
\label{sec:abstract-proof}

The estimate of Assumption \ref{ass:locapprox} on ``potentials" carries over
to $\ell$-forms with a discrete exterior derivative, that
is, the elements of the space $\HXt{\ell}{p}$, see \eqref{eq:Xptilde}.

\begin{lemma}{\rm (Global projection error estimate)}
  \label{lem:1}
  Making Assumptions \ref{ass:lift} through \ref{ass:locapprox}, the estimate
  \begin{equation*}
    \NFLtwo{\ell}{\bfu-\PO{\ell}{p}\bfu} \leq C\varepsilon_{\ell-1}(p)
    \NHX{\ell}{\bfu}\quad\forall \bfu\in \HXt{\ell}{p}
  \end{equation*}
  holds true, with a constant $C>0$ independent of $p$.
\end{lemma}

\begin{proof}
  Pick any $\bfu\in \HXt{\ell}{p}$. The locality of the projector $\PO{\ell}{p}$,
  \textit{cf.} \eqref{eq:POdef}, and \eqref{eq:NHXloc} allow purely local
  considerations.  Single out one cell $K\in\mesh$, still write
  $\bfu=\bfu\on{K}\in\HXt[K]{\ell}{p}$, and split $\bfu$ on $K$ using \eqref{eq:lift}
  from Assumption~\ref{ass:lift}:
\begin{equation}
    \label{eqo:8}
    \bfu = \extd[\ell-1]\Poinc_{\ell,K}\bfu + \Poinc_{\ell+1,K}\extd[\ell] \bfu = 
    \extd[\ell-1]\phibf + \Poinc_{\ell+1,K}\extd[\ell]\bfu \;.
\end{equation}
with $\phibf := \Poinc_{\ell,K} \bfu$. The continuity of $\Poinc_{\ell,K}$ from
Assumption \ref{ass:lift} reveals that
\begin{equation}
\label{eqo:9}
    \NHS[K]{\ell-1}{\phibf} \leq C \NHX[K]{\ell}{\bfu} \;,
\end{equation}
where here and below $C$ will denote constants (possibly different at different
occurrences) which depend neither on $\bfu$ nor on $p$. 

Thanks to identity \eqref{eqo:8} and the commuting diagram property 
from Assumption \ref{ass:cdp}, we have
\begin{equation}
\label{eqo:8b}
    \LPO[K]{\ell}{p}\bfu = \extd[\ell-1]\LPO[K]{\ell-1}{p}\phibf + 
    \LPO[K]{\ell}{p}\Poinc_{\ell+1,K}\extd[\ell]\bfu\;.
\end{equation}
Recall that $\bfu\in \HXt[K]{\ell}{p}$ belongs to the domain of $\LPO[K]{\ell}{p}$ by
Assumption \ref{ass:lp}.  Further, as $\bfu\in\HXt[K]{\ell}{p}$, from Assumption
\ref{ass:Wlift} we infer that
\begin{equation}
\label{eqo:73}
    \Poinc_{\ell+1,K} \extd[\ell]\bfu \in \Wpl[K]{\ell}{p}\;.
\end{equation}
Thus, owing to the identities \eqref{eqo:8}, \eqref{eqo:8b} and the projector
property of $\LPO[K]{\ell}{p}$, the task is reduced to an interpolation estimate for
$\LPO[K]{\ell-1}{p}$:
\begin{equation}
\label{eqo:10}
    (\Op{Id}-\LPO[K]{\ell}{p})\bfu =
    \extd[\ell-1] (\Op{Id}-\LPO[K]{\ell-1}{p})\phibf + 
    \underbrace{(\Op{Id}-\LPO[K]{\ell}{p})\Poinc_{\ell+1,K}\extd[\ell]\bfu}_{=0\;
    \text{by \eqref{eqo:73}}}\;.
\end{equation}
As a consequence, invoking Assumption \ref{ass:locapprox},
\begin{multline}
\label{eqo:11}
    \NFLtwo[K]{\ell}{(\Op{Id}-\LPO[K]{\ell}{p})\bfu} 
    \overset{\text{\eqref{eqo:10}}}{=}
    \NFLtwo[K]{\ell}{\extd[\ell-1](\Op{Id}-\LPO[K]{\ell-1}{p})\phibf} \\ \leq
    \varepsilon_{\ell-1}(p)\NHS[K]{\ell-1}{\phibf} 
    \overset{\text{\eqref{eqo:9}}}{\leq} 
    C\varepsilon_{\ell-1}(p) \NHX[K]{\ell}{\bfu}\;,
\end{multline}
which furnishes a local version of the estimate. This estimate is uniform in $K\in\mesh$ because $\mesh$ is finite. Due to \eqref{eq:NHXloc}, squaring
\eqref{eqo:11} and summing over all cells finishes the proof.
\end{proof}

We are now in the position to prove the main result of this section.
\begin{theorem}{\rm (Discrete compactness)}
  \label{thm:main}
  Under Assumptions \ref{ass:reg} through \ref{ass:locapprox}, the discrete
  compactness property of Definition \ref{def:dc} holds for the family
  $\big(\zbWpl[]{\ell}{p}\,\big)_{p\in\mathbb{N}}$ of subspaces of $\zbFHd{\ell}$. 
\end{theorem}

\begin{proof}
  The proof resorts to the ``standard policy'' for tackling the problem of discrete
  compactness, introduced by Kikuchi \cite{KIK89,KIK01} for analyzing the $h$-version
  of Whitney-1-forms. It forms the core of most
  papers considering the issue of discrete compactness, see \cite[Thm.~2]{BDC03},
  \cite[Thm.~11]{BCD06}, \cite[Thm.~4.9]{HIP02}, \cite[Thm.~2]{DMS00}, etc.

Let us introduce the discrete analogue of the space $\zbFXd{\ell}$:
\begin{equation}
  \label{eq:Xpldef}
  \zbXpl[]{\ell}{p} := \{\bfv_{p}\in \zbWpl[]{\ell}{p}:\;
  \SPN{\bfv_{p}}{\extd[\ell-1]\psibf_{p}} =  0 \;\;\forall \psibf_p\in \zbWpl[]{\ell-1}{p}\}.
\end{equation}
The space $\zbXpl[]{\ell}{p}$ contains $\zbZpl[]{\ell}{p}$ as a subspace.

We consider a subsequence $\N'$ of $\N$ and a $\FHd{\ell}$-bounded sequence $\left(\bfu_{p}\right)_{p\in\bbN'}$
with members in $\zbZpl[]{\ell}{p}$. Thus $\bfu_p$ belongs in particular to $\zbXpl[]{\ell}{p}$ and the sequence $\left(\bfu_{p}\right)_{p\in\bbN'}$ satisfies
\begin{align}
\label{eqo:12}
    \text{(i)} \quad & \bfu_{p} \in \zbWpl[]{\ell}{p}\;,\\
\label{eqo:13}
    \text{(ii)} \quad & \SPN{\bfu_{p}}{\extd[\ell-1]\psibf_{p}} =  0 
    \quad \forall \psibf_p\in \zbWpl[]{\ell-1}{p}\;,\\
\label{eqo:17}
    \text{(iii)} \quad & \NFHd{\ell}{\bfu_{p}} \leq 1\quad\forall p\in\bbN'\;.
\end{align}  
We have to confirm that it possesses a subsequence that converges in 
  $\FLtwo{\ell}$. 
  
  We start with the $\FLtwo{\ell}$-orthogonal projection of $\bfu_{p}$ into
  $\zbFXd{\ell}$ and parallel to $\extd[\ell-1]\zbFHd{\ell-1}$: let $\widetilde{\bfu}_{p}$
  be the unique vector field in $\zbFHd{\ell}$ with
\begin{align}
\label{eq:98}
   &\widetilde{\bfu}_{p} = \bfu_{p} + \extd[\ell-1]\widetilde{\phibf}_{p},
   \quad \widetilde{\phibf}_{p}\in \zbFHd{\ell-1}\;,\\
\label{eq:99}
   &\SPN{\widetilde{\bfu}_{p}}{\extd[\ell-1]\psibf} = 0 \quad\ 
   \forall \psibf\in \zbFHd{\ell-1}  \;.
\end{align}
Obviously, the latter condition implies 
\begin{equation}
\label{eqo:14}
      \widetilde{\bfu}_{p} \in \zbFXd{\ell}\;.
\end{equation}
Hence, by virtue of Assumption \ref{ass:reg}, the fact that $\extd[\ell]\widetilde{\bfu}_{p} =
\extd[\ell]\bfu_{p}$, and \eqref{eq:Xptilde}, $\widetilde{\bfu}_{p}$ satisfies
\begin{equation}
\label{eq:100}
      \widetilde{\bfu}_{p}\in \HXt{\ell}{p},\quad
      \NHX{\ell}{\widetilde{\bfu}_{p}} \leq C \NFHd{\ell}{\bfu_{p}}\,,
\end{equation}
where $C>0$ does not depend on $p$.

Since 
$\extd[\ell-1]\widetilde{\phibf}_{p}=\widetilde{\bfu}_{p}-\bfu_{p} \in \HX{\ell}$, Assumption~\ref{ass:Smax} implies that we may assume that 
$\widetilde{\phibf}_{p}\in\HS{\ell-1}$.

Thus we can use
N\'ed\'elec's trick \cite{NED80} to obtain
\begin{equation}
\label{eqo:15}
   \begin{aligned}
        \NFLtwo{\ell}{\widetilde{\bfu}_{p}-\bfu_{p}}^{2} & =
        \SPN{\widetilde{\bfu}_{p}-\bfu_{p}}{{\widetilde{\bfu}_{p}-
            \PO{\ell}{p}\widetilde{\bfu}_{p}+
            \PO{\ell}{p}\widetilde{\bfu}_{p}-\bfu_{p}
          }} \\
        & = \SPN{\widetilde{\bfu}_{p}-\bfu_{p}}{{\widetilde{\bfu}_{p}-
            \PO{\ell}{p}\widetilde{\bfu}_{p}}}\;.
   \end{aligned}
\end{equation}
This holds because from \eqref{eq:98} and the projector property of $\PO{\ell}{p}$ we
know
\begin{equation*}
   \PO{\ell}{p}\widetilde{\bfu}_{p} - \bfu_{p} =
   \PO{\ell}{p}\bfu_{p} + \PO{\ell}{p}\extd[\ell-1]\widetilde{\phibf}_{p} - \bfu_{p} =
   \PO{\ell}{p}\extd[\ell-1]\widetilde{\phibf}_{p} \,,
\end{equation*}
and thanks to the commuting diagram property \eqref{eq:globalcdp} 
(deduced from Assumption~\ref{ass:cdp}) combined with
the orthogonality conditions \eqref{eqo:13} and \eqref{eq:99}, we find
\begin{equation}
\label{eq:115}
    \SPN{\widetilde{\bfu}_{p}-\bfu_{p}}{
            \PO{\ell}{p}\widetilde{\bfu}_{p}-\bfu_{p} } =
    \SPN{\widetilde{\bfu}_{p}-\bfu_{p}}{
            \extd[\ell-1] \PO{\ell-1}{p} \widetilde{\phibf}_{p} } = 0\;.
\end{equation}
Hence, appealing to Lemma~\ref{lem:1}, with $C>0$ independent of $p$, we get
\begin{equation}
\label{eqo:16}
      \begin{aligned}
      \NFLtwo{\ell}{\widetilde{\bfu}_{p}-\bfu_{p}} 
        & \leq \NFLtwo{\ell}{\widetilde{\bfu}_{p}-\PO{\ell}{p}\widetilde{\bfu}_{p}} 
        \leq C \varepsilon_{\ell-1}(p)\NHX{\ell}{\widetilde{\bfu}_{p}} \\
        & \!\!\overset{\text{\eqref{eq:100}}}{\leq}
        C \varepsilon_{\ell-1}(p) \NHX{\ell}{\bfu_{p}} 
        \to 0 \quad\text{for }
        p\to\infty\;.
      \end{aligned}
\end{equation}
From \eqref{eq:100} we conclude that the sequence
${(\widetilde{\bfu}_{p})}_{p\in\bbN'}$ is uniformly bounded in $\HX{\ell}$.  By
Assumption \ref{ass:compemb} it has a convergent subsequence in $\FLtwo{\ell}$. Owing
to \eqref{eqo:16}, the same subsequence of ${(\bfu_{p})}_{p\in\mathbb{N}'}$ will
converge in $\FLtwo{\ell}$.
\end{proof}

\subsection{Approximation of the eigenvalue problem}
\label{sec:approx-evp}

As discussed in Section~\ref{sec:discrete-compactness}, the discrete compactness property is the cornerstone of the proof of the convergence of the discrete generalized Maxwell eigenvalue problem \eqref{eq:evp_p}.

\begin{corollary}
\label{cor:approx-evp}
In addition to the hypotheses of Theorem~\ref{thm:main}, namely
Assumptions~\ref{ass:reg} to \ref{ass:locapprox}, assume that 
property \emph{(CAS)} \eqref{eq:CAS} holds and that the space $\HX{\ell}\cap \zbkFHd{\ell}$ is
dense in $\zbkFHd{\ell}$. Then \eqref{eq:evp_p} provides a spectrally correct,
spurious-free approximation of the eigenvalue problem \eqref{eq:evp}.
\end{corollary}

\begin{proof}
  We use Theorem~\ref{thm:spurfree} from Section \ref{sec:discrete-compactness}.
  Considering that the discrete compactness property is provided by
  Theorem~\ref{thm:main}, and that we assume the approximation property (CAS)
  \eqref{eq:CAS}, we only need to show the approximation property (CDK) \eqref{eq:CDK},
  which concerns the approximation of closed forms by closed discrete forms.
  
  Since we assumed the density of $\HX{\ell}\cap \zbkFHd{\ell}$ in $\zbkFHd{\ell}$,
  it is sufficient to prove (CDK) for $\bfz\in \HX{\ell}\cap \zbkFHd{\ell}$.
  Such $\bfz$ belongs to $\HXt{\ell}{p}$, and we can therefore apply
  Lemma~\ref{lem:1}, which shows that $\PO{\ell}{p}\bfz\to\bfz$ in $\FLtwo{\ell}$. We
  will have accomplished to show (CDK) with $\bfz_p=\PO{\ell}{p}\bfz$, as
  soon as we show that $\extd[\ell]\bfz_p=0$. Keeping in mind that
  $\bfz_p\in\zbWpl[]{\ell}{p}\subset\FHd{\ell}$, we see that it is sufficient to show
  the local relation $\extd[\ell]\bfz_p=0$ in $K$ for every $K\in\mesh$. This follows
  finally as in \eqref{eqo:8b} in the proof of Lemma~\ref{lem:1}, because
  $\extd[\ell]\bfz=0$ implies
\[
 \LPO[K]{\ell}{p}\bfz = \extd[\ell-1]\LPO[K]{\ell-1}{p} \Poinc_{\ell,K} \bfz\,.
\]
Hence $\extd[\ell]\bfz_p = \extd[\ell]\LPO[K]{\ell}{p}\bfz = 
\extd[\ell]\extd[\ell-1]\LPO[K]{\ell-1}{p} \Poinc_{\ell,K} \bfz = 0$, which ends the proof.
\end{proof}

  \begin{remark}
    The abstract theory developed in this section can be applied to the $h$-version
    of discrete differential forms, if the dependence of the constants on the size of
    the cell $K$ is made explicit by means of scaling arguments. Here, we forgo this
    extra technicality and refer the reader to \cite[Sect.~4.4]{HIP02}.
  \end{remark}

\section{Regularized Poincar\'e lifting}
\label{sec:smooth-poic-lift}

In this section we describe the construction of a local lifting operator
$\Poinc_{\ell}$ that will satisfy Assumptions~\ref{ass:lift} and \ref{ass:Wlift} in
Section~\ref{sec:local-lifting} for suitable spaces $\HX[K]{\ell}$, $\HS[K]{\ell}$
and $\Wpl[K]{\ell}{p}$. We follow the presentation in \cite{CMI08}, where these
operators are analyzed and where it is shown in particular that they are
pseudodifferential operators of order $-1$.

\subsection{Definition}
\label{sec:def-of-Poinc}

We consider a bounded domain $D\subset \R^d$ that is \emph{star-shaped}
with respect to some subdomain $B\subset D$, that is,
\begin{gather}
  \label{eq:starB}
  \forall a\in B,\, x\in D:\quad 
  \{(1-t) a + tx,\; 0<t<1\} \subset D\;.
\end{gather}

For $a\in B$ and $1\le\ell\le d$, we define the \emph{Poincar\'e operator}
$\Poinc_{\ell,a}$, acting on a differential form $\bfu\in \FCinf[D]{\ell}$, by the
path integral
\begin{gather}
\label{eq:Poincla}
 \Poinc_{\ell,a}\bfu(x) = 
 (x-a) \contr \int_0^1 t^{\ell-1}\,\bfu\bigl(a+t(x-a)\bigr)\,dt\,,
   \qquad x\in D\,.
\end{gather}
Here the symbol $\contr$ denotes the contraction (also called ``interior product'')
of the vector field $x\mapsto(x-a)$ with the $\ell$-form $\bfu$.
It is clear that $\Poinc_{\ell,a}$ maps $\FCinf[D]{\ell}$ to $\FCinf[D]{\ell-1}$ and
it has been shown (see \cite{GD04} for proofs in the case $d=3$) that it can be
extended to a bounded operator from $\FLtwo[D]{\ell}$ to $\FLtwo[D]{\ell-1}$.  In
order to define the \emph{regularized Poincar\'e operator} $\Poinc_{\ell}$, we choose
a function
\[
  \theta\in C_0^\infty(\R^d)\,, \quad \supp \theta \subset B\,, \quad
  \int_B\theta(a)\,da=1\,,
\]
and set
\begin{gather}
\label{eq:Poincl}
  \Poinc_{\ell}\bfu(x) = \int_B \theta(a) \Poinc_{\ell,a}\bfu(x)\,da\,.
\end{gather}

\subsection{Regularity}
\label{sec:regularity-of-Poinc}

The substitution
$y=a+t(x-a)$, $\tau=1/(1-t)$ transforms the double integral in \eqref{eq:Poincla}, \eqref{eq:Poincl} into 
\begin{gather}
  \label{eq:104}
  \begin{aligned}
  \Poinc_{\ell}\bfu(x) & = \int\limits_{\mathbb{R}^{d}}\int\limits_{1}^{\infty}
  (\tau-1)^{\ell-1}\tau^{d-\ell}
          \theta\bigl(x+\tau(y-x)\bigr)\, (x-y) \contr\bfu(y) \,d\tau\,dy \\
  & = \int\limits_{\mathbb{R}^{d}}  k(y,y-x) \contr\bfu(y) \,dy\;,
  \end{aligned}
\end{gather}
where the kernel $k(y,z)$ has an expansion into quasi-homogeneous terms:
\begin{gather}
  \label{eq:kernel}
  \begin{aligned}
    k(y,z) & = 
    -z \int_0^\infty s^{\ell-1} (s+1)^{d-\ell}
    \theta\bigl(y+s z\bigr)\,ds \\
    &=
    - \sum_{j=0}^{d-\ell}\tbinom{d-\ell}{j} \frac{z}{|z|^{d-j}}
     \int_0^\infty r^{d-j-1} \theta\bigl(y+r \frac{z}{|z|}\bigr)\,dr
\;.
  \end{aligned}
\end{gather}
The operator $\Poinc_{\ell}$ is therefore a weakly singular integral operator. In \cite[Section 3.3]{CMI08}, the following result is shown.

\begin{proposition}
\label{pro:pseudo}
For $1\le\ell\le d$, the operator $\Poinc_{\ell}$ is a pseudodifferential operator of
order $-1$ on $\R^d$.  It is well defined on $\FCinf[D]{\ell}$, it maps
$\FCinf[D]{\ell}$ to $\FCinf[D]{\ell-1}$ and $\FCinf[\overline{D}]{\ell}$ to
$\FCinf[\overline{D}]{\ell-1}$, and for any $s\in\R$ it has an
extension as a bounded operator
\[
  \Poinc_{\ell}\,:\: 
    \FHm[D]{\ell}{s} \to \FHm[D]{\ell-1}{s+1})\;.
\] 
Here, $\FHm[D]{\ell}{s}$ is the Sobolev space of $\ell$-forms on $D$ of order $s$.
\end{proposition}

\subsection{Lifting property}
\label{sec:lifting-Poinc}

The lifting property \eqref{eq:lift} is a consequence of the following identity,
which is a special case of ``Cartan's magic formula'' for Lie derivatives and for a
flow field generated by the dilations with center $a$.
\begin{multline}
\label{eq:Cartan}
 \frac{d}{dt}(t^\ell \bfu\bigl(a+t(x-a)\bigr) =\\
    \extd[\ell-1]\Bigl(t^{\ell-1}(x-a)\contr \bfu\bigl(a+t(x-a)\bigr)\Bigr) +
    t^{\ell}(x-a)\contr \extd[\ell]\bfu\bigl(a+t(x-a)\bigr)
\end{multline}
Here $\bfu$ is an $\ell$-form. The result is
\begin{equation}
\begin{aligned}
\label{eq:dR+Rd=1}
 \extd[\ell-1]\Poinc_\ell \bfu + \Poinc_{\ell+1} \extd[\ell]\bfu &= \bfu\; &&(1\le\ell\le d-1)\;;\\
 \Poinc_1 \extd[0]\bfu &= \bfu - \bigl(\theta,\bfu\bigr)_{0,D}\; &&(\ell=0)\;;\\
 \extd[d-1]\Poinc_d\bfu &=\bfu &&(\ell=d)\;.
\end{aligned}
\end{equation}
These relations are valid for all $\bfu\in C_0^{\infty}(\R^d,\Lambda^\ell)$ and by extension for all
$\bfu\in H^{s}(D,\Lambda^{\ell})$, $s\in\R$.

The perfect match of \eqref{eq:dR+Rd=1} with \eqref{eq:lift} from Assumption
\ref{ass:lift} suggests that the regularized Poincar\'e lifting $\Poinc_{\ell}$ provides
suitable local liftings as stipulated in Assumption \ref{ass:lift}. To this end,
we can choose as local spaces of ``more regular forms'' 
\begin{equation}
  \label{eq:XS}
  \begin{gathered}
    \HX[K]{\ell} := \FHd[K]{\ell}\cap \FHm[K]{\ell}{r}\;,\\
    \HS[K]{\ell-1} := \FHdr[K]{\ell-1}{r}
    \quad\mbox{and}\quad
    \HS[K]{\ell} := \FHdr[K]{\ell}{r}  \;,
  \end{gathered}
\end{equation}
for some $0<r\leq 1$, where we denote by $\FHdr[K]{k}{r}$ the space
\[
   \FHdr[K]{k}{r} 
   := \{\bfv\in \FHm[K]{k}{r}: \extd[k]\bfv\in\FHm[K]{k+1}{r}\}\;.
\]
All these spaces are equipped with the natural Hilbert space norms. Also keep in mind
that the global spaces $\HX{\ell}$, $\HS{\ell-1}$ and $\HS{\ell}$ are determined by
their local definition on the mesh cells $K$, \textit{cf.} \eqref{eq:XK} and \eqref{eq:SK}.
For the particular choice \eqref{eq:XS} an assumption of
Corollary~\ref{cor:approx-evp} can be verified.

\begin{lemma}
  \label{lem:dense}
  For $\HX{\ell}$ arising from \eqref{eq:XS} the space $\HX{\ell}\cap \zbkFHd{\ell}$
  is dense in $\zbkFHd{\ell}$.
\end{lemma}

\begin{proof}
  By \cite[Thm.~4.9(c)]{CMI08} we have a direct decomposition
\begin{equation}
\label{eq:dec1}
    \zbkFHd{\ell} = \extd[\ell-1]\zbFHm{\ell-1}{1} \,\oplus\, \mathcal{C}_{\ell}\;,\quad
    \mathcal{C}_{\ell} \subset C^{\infty}_{\overline\Omega}(\R^n,\Lambda^\ell)\;,
\end{equation}
  where $C^{\infty}_{\overline\Omega}(\R^d,\Lambda^\ell)$ is the space of compactly
  supported, smooth $\ell$-forms on $\R^d$ with support contained in
  $\overline\Omega$ or, equivalently, the space of all smooth $\ell$-forms on
  $\overline\Omega$ that vanish on $\partial\Omega$ together with all their
  derivatives.  Since $C^{\infty}_{\overline\Omega}(\R^d,\Lambda^{\ell-1})$ is dense in 
  $\zbFHm{\ell-1}{1}$, we deduce: 
\[ C^{\infty}_{\overline\Omega}(\R^d,\Lambda^\ell)\cap
  \extd[\ell-1]\zbFHm{\ell-1}{1} \quad \mbox{is dense in}\quad
  \extd[\ell-1]\zbFHm{\ell-1}{1}
\]
 As every
  $\bfu\in C^{\infty}_{\overline\Omega}(\R^d,\Lambda^\ell)$ belongs to $\HX{\ell}$, the
  assertion follows.
\end{proof}

We point out that the choice of $r$ in \eqref{eq:XS} is determined by Assumption
\ref{ass:reg}. Also note that whenever we opt for \eqref{eq:XS}, Rellich's theorem
ensures Assumption \ref{ass:compemb}, because the mesh is kept fixed.

The construction of $\Poinc_{\ell}$ entails a constraint on the cell shapes. This is
satisfied for standard finite element meshes, where the cells usually are convex
polyhedra.
\medskip

\ass{ass:K}{Every cell $K\in\mesh$ is a star-shaped polyhedron.}

\begin{lemma}
\label{lem:assass}
  Assumption \ref{ass:K}, the choice \eqref{eq:XS} for spaces $\HX[K]{\ell}$ and $\HS[K]{\ell-1}$ imply
  Assumptions \ref{ass:compemb}, \ref{ass:Smax} and \ref{ass:lift}.
\end{lemma}

\begin{proof}
The only fact remaining to be proved is the maximality relation in Assumption \ref{ass:Smax}
\[
   \extd[\ell-1] \HS{\ell-1} = \extd[\ell-1] \zbFHd{\ell-1} \cap \HX{\ell}.
\]
The inclusion $\subset$ holds by definition. Let us prove the converse inclusion.\\
Let $\bfu\in\extd[\ell-1] \zbFHd{\ell-1} \cap \HX{\ell}$. Thus $\bfu=\extd[\ell-1]\phibf$ with $\phibf\in\zbFHd{\ell-1}$. Since $\bfu\in\FLtwo{\ell}$, using \cite[Cor.\ 4.7]{CMI08} we obtain that there exists $\psibf\in\zbFHm{\ell-1}{1}$ such that $\bfu=\extd[\ell-1]\psibf$. In particular, $\psibf\on{K}$ belongs to $\FHm[K]{\ell-1}{r}$ for all $K$ and, since $\bfu\on{K}$ belongs to $\FHm[K]{\ell }{r}$, we finally find that $\psibf\on{K}\in\FHdr[K]{\ell-1}{r}$.
\end{proof}

\subsection{Preservation of polynomial forms}
\label{sec:polynomials-and-Poinc}

Fundamental in finite element methods is the notion of polynomial differential forms.
For an ordered $\ell$-tuple $I=(i_{1},\ldots,i_{\ell})$,
$i_{1}<i_{2}<\ldots<i_{\ell}$, $\{i_{1},\ldots,i_{\ell}\}\subset\{1,\ldots,d\}$, let
\begin{gather*}
  \mathsf{d} x_{I} := \mathsf{d}x_{i_{1}} \wedge \cdots \wedge\mathsf{d}x_{i_{\ell}} \;,
\end{gather*}
where $\mathsf{d}x_{j}$, $j=1,\ldots,d$, are the co-ordinate 1-forms in Euclidean
space $\bbR^{d}$. The space $\FPol{p}{\ell}$ of polynomial $\ell$-forms on $\bbR^{d}$
is defined as
\begin{gather*}
  \FPol{p}{\ell} := \Big\{\bfu = \sum\nolimits_{I}u_{I}\,\mathsf{d} x_{I}:\, 
  u_{I}\in\Pol_{p}(\bbR^{d})\Big\}\;,
\end{gather*}
where $\sum_{I}$ indicates summation over all ordered $\ell$-tuples, and
$\Pol_{p}(\bbR^{d})$ is the space of $d$-variate polynomials of total degree $\leq p$. 
We remark that for $d\in \{2,3\}$ polynomial forms possess polynomial vector
proxies. 

From the definition \eqref{eq:Poincla} it is clear that the Poincar\'e operator
$\Poinc_{\ell,a}$ maps differential forms with polynomial coefficients to
differential forms with polynomial coefficients. The same holds for the regularized
Poincar\'e operator $\Poinc_{\ell}$ by \eqref{eq:Poincl}. If we want $\Poinc_{\ell}$
to map a space $P(\Lambda^\ell)$ of differential forms of order $\ell$ (e.g., with
polynomial coefficients) into a space $P(\Lambda^{\ell-1})$ of differential forms of
order $\ell-1$, it is sufficient to require the following two properties, see
\cite[Proposition 4.2]{CMI08}.

\begin{proposition}
\label{pro:poly-Poinc}
Assume that $P(\Lambda^\ell)$ and $P(\Lambda^{\ell-1})$ are finite-dimensional spaces of differential forms satisfying \\
\iti1 The space $P(\Lambda^{\ell})$ is invariant with respect to dilations and
translations, that is
$$
   \text{For any }t\in\R, a\in\R^n:
   \text{ if }\bfu\in P(\Lambda^{\ell}), \text{ then }
   \bigl(x\mapsto \bfu(tx+a)\bigr) \in P(\Lambda^{\ell}) \;.
$$ 
\iti2
The interior product
 $x\contr:\bfu\mapsto x\contr \bfu$ maps 
 $P(\Lambda^{\ell})$ to $P(\Lambda^{\ell-1})$.\\
Then $\Poinc_{\ell}$ maps  $P(\Lambda^\ell)$ into $P(\Lambda^{\ell-1})$. 
\end{proposition}

For the compatibility Assumption~\ref{ass:Wlift} to hold, it is therefore sufficient
to make the following assumption about the local polynomial space $\Wpl[K]{\ell}{p}$.
\medskip

\ass{ass:poly-Poinc}{
 \mbox{}
  \begin{enumerate}
    \renewcommand{\labelenumi}{(\roman{enumi})}
  \item 
    The space $\Wpl[K]{\ell}{p}$ is invariant with respect to dilations and
    translations.
  \item
    The differential operator
    $x\contr\extd[\ell]:\bfu\mapsto x\contr \extd[\ell]\bfu$ maps 
    $\Wpl[K]{\ell}{p}$ into $\Wpl[K]{\ell}{p}$.  
\end{enumerate}}

\medskip 
To summarize:
\begin{center}
  Assumptions \ref{ass:K}, \ref{ass:poly-Poinc}, and \eqref{eq:XS}\quad
  $\Longrightarrow$\quad
  Assumptions  \ref{ass:compemb}, \ref{ass:Smax}, \ref{ass:lift}, and \ref{ass:Wlift}.
\end{center}

\section{Discrete differential forms}
\label{sec:discr-diff-forms-2}

Now we introduce concrete spaces of discrete differential forms. We merely summarize
the constructions that have emerged from research in differential geometry (the
``Whitney-forms'' introduced in \cite{WIT57}) and finite element theory
(``Raviart-Thomas elements'' of \cite{RAT77} and ``N\'ed\'elec finite elements'' of
\cite{NED80,NED86}). These schemes were later combined into the concept of discrete
differential forms \cite{BOS88a,HIP96b}. Surveys and many more details can be found
in \cite{HIP02,AFW06,AFW09,BOS05c}.

\subsection{Simplicial meshes}
\label{sec:simplicial-meshes}

Let $\mesh$ be a conforming simplicial finite element mesh covering the Lipschitz polyhedron
$\Omega\subset\bbR^{d}$. As elaborated in \cite[Sect.~3 \& 4]{AFW06} for
$p\in\mathbb{N}$ the following choices
\begin{align}
  \label{eq:ned1}
  \Wpl[K]{\ell}{p} & := \FPol{p-1}{\ell}\on{K} +  x \contr
  \FPol{p-1}{\ell+1}\on{K}\\[-2.0ex]
  \intertext{and} 
  \label{eq:ned2}
  \Wpl[K]{\ell}{p} & := \FPol{p}{\ell}\on{K}
\end{align}
of local spaces, through \eqref{eq:Wlploc}, gives rise to meaningful global finite
elment spaces $\zbWpl{p}{\ell}$ of discrete differential forms.

By construction both Assumption \ref{ass:K} and Assumption \ref{ass:poly-Poinc} are
satisfied for these spaces. The asymptotic density property also holds.

\begin{lemma}
  \label{lem:Vpdense}
  The spaces $\zbWpl{p}{\ell}$ of discrete differential forms built from
  \eqref{eq:ned1} or \eqref{eq:ned2} meet the requirement \eqref{eq:CAS}.
\end{lemma}

\begin{proof}
  It is a classical result of finite element theory that the spaces
  of degree $p$ Lagrangian finite element functions $\zbWpl{0}{p}$
  are asymptotically dense in $\zbHm{1}$. Thus the space of 
  polynomials $\ell$-forms with coefficients in $\zbWpl{0}{p}$, which is
  a subspace of $\zbWpl{\ell}{p}$, is asymptotically dense in $\zbFHm{\ell}{1}$. The latter
  space is obviously dense in $\zbFHd{\ell}$, since this is already true
  for $C^{\infty}_{0}(\Omega,\Gamma^{l})$.
\end{proof}

\subsection{Tensor product meshes}
\label{sec:tens-prod-mesh}

Let $\mesh$ be a conforming finite element mesh of the Lipschitz polyhedron $\Omega$ whose cells are affine
images of the unit hypercube $\widehat{K}$ in $\bbR^{d}$: for $K\in\mesh$ the we
write $\Phibf_{K}:\widehat{K}\mapsto K$ for the associated unique affine mapping.  We
generalize the construction of \cite{NED80}: on the cube we define (with notations introduced in Section \ref{sec:polynomials-and-Poinc})
\begin{gather*}
  \Wpl[\widehat{K}]{\ell}{p} := \Big\{
  \widehat{\bfv} = \sum\limits_{I} u_{I}\mathsf{d}x_{I},\,
  u_{I}(x) = \prod\limits_{j=1}^{d}u_{I,j}(x_{j}),\,
  u_{I,j}\in 
  \begin{cases}
    \Pol_{p-1} \!\!\! & \text{if } j\in I\\
    \Pol_{p} & \text{if } j\not\in I
  \end{cases}
  \ \Big\}\;.
\end{gather*}
The local spaces are obtained by affine pullback
\begin{gather}
  \label{eq:wpq}
  \Wpl[K]{\ell}{p} := \bigl(\Phibf_{K}^{-1}\bigr)^{\ast} \Wpl[\widehat{K}]{\ell}{p}\;.
\end{gather}
This affine tensor product construction also complies with Assumption \ref{ass:K} and
Assumption \ref{ass:poly-Poinc}. Completely parallel to Lemma~\ref{lem:Vpdense}, one
proves the following result.

\begin{lemma}
  \label{lem:Qpdense}
  The requirement \eqref{eq:CAS} is satisfied for the spaces $\zbWpl{p}{\ell}$
  spawned by \eqref{eq:wpq}. 
\end{lemma}

  \begin{remark}
    For all the above meshes the cells are affine images of a single reference cell,
    the ``unit simplex'' or ``unit hypercube''. We could allow some non-affine cells:
    Under the assumption that the transformations are ``nearly affine'', see
    \cite[\S 4.3]{CIA78}, and the projection operators $\pi^{\ell}_{p,K}$ are 
     defined correspondingly,
    all crucial estimates like Lemma~\ref{lem:1} can be
    transferred to the reference cell using the pullback of differential forms.
  \end{remark}

\section{Application in dimensions two and three}
\label{sec:discr-diff-forms}

We adopt the discrete spaces from Sect.~\ref{sec:discr-diff-forms-2} along with the
regularized Poincar\'e lifting from Sect.~\ref{sec:smooth-poic-lift}. We rely on the
choice \eqref{eq:XS} for spaces $X$ and $S$, with a regularity exponent $r\in(0,1]$
which has to be chosen suitably.

In order to establish the discrete compactness property from Definition \ref{def:dc},
it remains to verify the regularity Assumption \ref{ass:reg} and Assumptions \ref{ass:lp},
\ref{ass:cdp}, and \ref{ass:locapprox} for convenient local projectors
$\LPO[K]{\ell}{p}$.

Local projectors which make the discrete diagram of Assumption \ref{ass:cdp} commute
do exist in the general framework of differential forms of any degree. They generalize
N\'ed\'elec edge element projections and can be referred to as \emph{moment based}
projection operators. They are suitable for the $h$-version of finite elements in
dimensions 2 and 3. In higher dimensions some of them (for low degree forms)
require a higher regularity than $H^2$ to be defined. In \cite{AFW06,AFW09},
they are modified by an extension-regularization procedure in order to be defined and
bounded on $L^2$. However, such operators cannot be used for the $p$ version of
finite elements, because no estimates (stability or error bounds) are known with
respect to the polynomial degree $p$.

The proper projection operators for $p$-version approximation are so-called
\emph{projection based interpolation operators}, see
\cite{DEB01,DEB04,CAD05,DEM06,DEK07}. Variants for any $\ell$ and $d$ are available
and they are designed to commute in the sense of Assumption \ref{ass:cdp}
\cite[Sect.~3.5]{HIP02}.

At this point we have to abandon the framework of general $\ell$ and $d$,
because both regularity results (Assumption \ref{ass:reg}) and the analysis of
projection operators (Assumption \ref{ass:locapprox}) are not presently available for
general $\ell$ and $d$. We have to discuss them for special choices of $\ell$ and $d$
separately, relying on a wide array of sophisticated results from the literature.

\begin{theorem}{\rm (Convergence of Galerkin approximations)}
  \label{thm:mainMax}
  For $d=2$, $3$, and $0\leq \ell<d$, the Galerkin discretization of \eqref{eq:evp}
  on a Lipschitz polyhedron based on any of the families of discrete
  differential forms introduced in Sect.~\emph{\ref{sec:discr-diff-forms-2}} offers a
  spectrally correct, spurious-free approximation.
\end{theorem}

\begin{proof}
  We skip the case $\ell=0$, for which the standard Galerkin approximation theory
  for operators with compact resolvent can be applied, see \cite{KNO06a}.
  
  To begin with, we focus on the discrete compactness property and verify the
  assumptions \ref{ass:reg}, \ref{ass:lp}, \ref{ass:cdp}, and \ref{ass:locapprox}
  for $d=2$ and $d=3$ separately. 

\smallskip\noindent$\bullet$ \  
$d=2$, $\ell=1$: in terms of vector proxies we find the correspondence
  \begin{gather}
    \label{eq:Yd2l1}
    \zbFXd{1} \sim \zbHcurltd \cap \kHdiv\;. 
  \end{gather}
  Regularity theorems for boundary value problems for $-\Delta$ on the polygon confirm
  the existence of $\delta=\delta(\Omega)>0$ such that
  \begin{gather}
    \label{eq:regYd2l1}
    \zbHcurltd \cap \kHdiv \subset \HHm{\delta+1/2}\;,
  \end{gather}
  in the sense of continuous embedding, see \cite[Sect.~3.2]{GIR86}. This suggests to
  choose $r=\delta+1/2$ in \eqref{eq:XS} and Assumption \ref{ass:reg} will hold
  true. Hence, we deal with the concrete spaces
  \begin{align}
    \label{hx:1}
    \HX{1} & = \zbHcurltd \cap \prod\limits_{K\in\mesh} 
    (\Hm[K]{\delta+1/2})^{2}\;,\\
    \label{hs:2}
    \HS{0} & = \zbHm{1} \cap \prod\limits_{K\in\mesh} 
    \Hm[K]{\delta+3/2}\;.
  \end{align}
  Commuting local projection based interpolation operators $\LPO[K]{1}{p}$ and
  $\LPO[K]{0}{p}$ have been proposed for triangles and for
  quadrilaterals in \cite{DEB01}. With the choice \eqref{hx:1} and \eqref{hs:2} they
  live up to Assumptions \ref{ass:lp} and \ref{ass:cdp}. Assumption
  \ref{ass:locapprox} holds with $\varepsilon_{0}(p) = C{p}^{-1/2}$ and $C>0$
  depending only on the shape-regularity of the cells, \textit{cf.} \cite[Thm.~4.3]{DEM06} and 
  \cite[Thm.~4.1]{BEH09a}. Finally, these interpolation operators satisfy the 
  natural condition of conformity \eqref{eq:LPO} by construction, which makes
  they meet all our requirements, \textit{cf.} Sect.~\ref{sec:local-projectors}.

\smallskip\noindent$\bullet$ \  
  $d=3$, $\ell=1,2$: we have the vector proxy incarnation
   \begin{gather}
     \label{eq:Yd3}
     \zbFXd{\ell} \sim 
     \begin{cases}
       \zbHcurl \cap \kHdiv & \text{for } \ell =1\;,\\
       \zbHdiv \cap \kHcurl & \text{for } \ell =2\;.
     \end{cases}
   \end{gather}
   Citing results from \cite{ABD96} and \cite[Sect.~4.1]{HIP02}, we find
   $\delta=\delta(\Omega)\in ]0,\frac{1}{2}]$ and continuous embeddings
  \begin{equation}
    \label{eq:regXmw}
    \zbHcurl \cap \Hdiv, \; \zbHdiv \cap \Hcurl \subset \HHm{\delta+1/2}\;.
  \end{equation}
  Therefore, using the construction \eqref{eq:XS} with $r=\delta+1/2$,
  Assumption \ref{ass:reg} is satisfied for $\ell\in\{1,2\}$. The relevant
  spaces of more regular forms now read
  \begin{align}
    \label{hxd3:1}
    \HX{1} & = \zbHcurl \cap \prod\limits_{K\in\mesh} 
    (\Hm[K]{\delta+1/2})^{3}\;,\\
    \label{hxd3l2:1}
    \HX{2} & = \zbHdiv \cap \prod\limits_{K\in\mesh} 
    (\Hm[K]{\delta+1/2})^{3}\;,\\
    \label{hsd3:2}
    \HS{0} & = \zbHm{1} \cap \prod\limits_{K\in\mesh} 
    \Hm[K]{\delta+3/2}\;,\\
    \label{hsd3l2:2}
    \HS{1} & = \zbHcurl \cap \prod\limits_{K\in\mesh} 
    \Hcurlm[K]{\delta+1/2}\;.
  \end{align}

  The essential commuting local projection based interpolation operators
  $\LPO[K]{m}{p}$, $m=0,1,2$, have been introduced in \cite{DEB04} for tetrahedral
  meshes and in \cite{DEB01} for meshes comprising parallelepipeds. By construction they
  comply with Assumption \ref{ass:cdp}. Assumption \ref{ass:lp} for the spaces 
  $\HS{0}$ and $\HS{1}$ from \eqref{hsd3:2} and \eqref{hsd3l2:2}, respectively,
  and $r=\delta+1/2$ is a consequence of Sobolev embedding
  theorems. Relying on \cite[Th.5.3]{DEM06} we obtain like in the 2D case that in
  Lemma~\ref{lem:1} we can
  take $\varepsilon_{m}(p) = Cp^{-1/2}$ for $m=0$ and $m=1$.
  
\smallskip\noindent$\bullet$ \  
  Finally, we appeal to Lemmas~\ref{lem:Vpdense}, \ref{lem:Qpdense} together with
  Lemma~\ref{lem:dense} and apply the abstract theory of
  Sect.~\ref{sec:generic-assumptions} in the form of Corollary~\ref{cor:approx-evp} to conclude the proof of the theorem.
\end{proof}  

\begin{corollary}{\rm (Approximation of the Maxwell eigenvalue problem)}
\label{cor:mainMax}
 The $p$ version finite element discretization of the Maxwell eigenvalue problem \eqref{eq:maxevp} based on 
 edge elements from the first or second N\'ed\'elec family on triangles or on tetrahedra, or from the first N\'ed\'elec family on parallelograms or on parallelepipeds 
 offers a spectrally correct, spurious-free approximation.
\end{corollary}

\begin{remark}
\label{rem:epsilonmu}
Instead of \eqref{eq:maxevp} we may consider the variational formulation of the more general Maxwell eigenvalue problem \eqref{eq:maxwell}, corresponding to the case of anisotropic inhomogeneous material:
\begin{gather}
  \label{eq:maxevpepsmu}
\aligned
   &\mbox{Seek $\bfu\in\zbHcurl\setminus\{0\}$, \ $\omega\in\R^+_0$ \ such that}\\
   &\SPLtwo{\mu^{-1}\curl\bfu}{\curl\bfv} = 
   \omega^{2}\SPLtwo{\epsilon\bfu}{\bfv}\ \forall
   \bfv\in\zbHcurl\;,
\endaligned
\end{gather}
with uniformly positive material tensors $\mu=\mu(\bfx)$,
$\epsilon=\epsilon(\bfx)$. The same edge element discretizations listed in Corollary~\ref{cor:mainMax}  provide spectrally correct, spurious-free approximations of this problem.
This generalization of Corollary~\ref{cor:mainMax} can be achieved with standard tools (see, in particular, Propositions~2.25,
  2.26, and~2.27 of~\cite{CFR00}, and \cite[Sect.~6]{HIP08t}, \cite[Thm.~4.9]{HIP02}).
\end{remark}

\begin{remark}
The restriction on the families of elements mentioned in the Corollary is essentially due to the availability of published results about suitable interpolation operators. Thus, for example, as soon as a generalization of the $p$ version error estimates of \cite{DEB04,DEM06} for projection-based interpolants to meshes containing prismatic or more general polyhedral elements becomes available, our result about the approximation of the Maxwell eigenvalue problem will apply to such meshes, too.    
\end{remark}

  \begin{remark}
    Several obstacles prevent us from establishing the assumptions of the abstract
    theory for $d>3$. On the one hand, continuity properties of projection based
    interpolation operators have not been investigated for $d>3$. Also, regularity
    results along the lines of \eqref{eq:regXmw} are have not been published for
    polyhedra in higher dimensions.

    On the other hand, the innocuously looking requirement \eqref{eq:LPO} for the
    projection operators --- corresponding to the requirement that the global projection operators are constructed elementwise from local degrees of freedom ---
    entails that the trace of forms in $S(K,\Lambda^{\ell-1})$
    onto $\ell-1$-dimensional facets in $\Faces[K]{\ell-1}$ must make sense. However, we
    cannot expect more than $H^2$ regularity for the space $\HS[K]{\ell-1}$. Hence, by
    trace theorems for Sobolev spaces, the spaces $\HS[K]{\ell-1}$ allow for
    traces on $m$-facets for $m>\frac{d}{2}-2$ at best, which means that $\ell
    > \frac{d}{2}-1$ is required to allow for the construction of a local
    projection based interpolation complying with Assumptions~\ref{ass:lp}
    and~\ref{ass:locapprox}.

    Perhaps, an analysis in $L^{p}$-spaces as in \cite[Lemma~4.7]{ABD96} can make 
    possible an extension of the theory to higher dimensions, but this is beyond
    the scope of the present article. 
  \end{remark}

\begin{remark}
Our approach does not cover $hp$-refinement, for various reasons. 
One reason is that there exist many variants of $hp$ refinements in 3D, and covering them would in any case require a much longer paper than the present one.

Another reason is technical:
The existing convergence proof of the $hp$ approximation of the Maxwell eigenvalue problem in \cite{BCD06} --- while also based on the proof of the discrete compactness property --- uses a different technical tool, namely an estimate of the $L^2$ stability of a certain projection operator. This kind of estimate is currently only available for intervals in 1D and for rectangles in 2D. 

The technique used in the present paper is based on the regularized Poincar\'e lifting, and adjusting this 
to variable polynomial degree poses formidable technical
  challenges. Only in 2D these could be mastered so far, as was demonstrated in \cite{BHH09}
  in the context of boundary element analysis.
\end{remark}


\section{Conclusion}
\label{sec:conclusion}

In this paper we have proved that the $p$-version of finite elements based on generalized N\'ed\'elec edge elements provides a spurious-free spectrally correct approximation of the Maxwell eigenvalue problem. The essential point was the proof of the discrete compactness property. We showed that this property follows from a set of rather natural assumptions about the family of finite element spaces and interpolation operators, and in addition we showed that these assumptions are implied by recently found results on lifting operators and on projection-based interpolants.

In the approach pursued in \cite{AFW06,AFW09} the discrete compactness property is
not addressed directly: in the framework of the $h$-version for differential forms,
modified moment-based projection operators are used. These new operators satisfy the
strong property of being uniformly bounded in $L^2$ and are constructed by means of a
delicate extension-regularization procedure, see also \cite{ChrWi08,Chr09}.

On the one hand this uniform boundedness property is stronger than our assumption
\ref{ass:lp} and replaces in a certain way the discrete compactness property. But on
the other hand, it is currently not known whether a construction of projection
operators by extension-regularization could also be employed in the case of the
$p$-version of finite elements, or whether the construction of a $p$-uniformly
$L^2$-bounded family of cochain projections is even possible.



\end{document}